\documentclass[hidelinks,onefignum,onetabnum]{siamart220329}

\usepackage{lipsum}
\usepackage{amsfonts}
\usepackage{graphicx}
\usepackage{epstopdf}
\usepackage{algorithmic}
\ifpdf
  \DeclareGraphicsExtensions{.eps,.pdf,.png,.jpg}
\else
  \DeclareGraphicsExtensions{.eps}
\fi

\newsiamremark{remark}{Remark}
\newsiamremark{hypothesis}{Hypothesis}
\crefname{hypothesis}{Hypothesis}{Hypotheses}
\newsiamthm{claim}{Claim}

\headers{Rational Approximation via the EIM}{A.~Li and Y.~Li}

\title{A new rational approximation algorithm via the empirical interpolation method\thanks{Submitted to the editors DATE.
\funding{This work was supported by the National Science Foundation of China (no.~12471346) and the Fundamental Research Funds for the Zhejiang Provincial universities (no. 226-2023-00039).}}}

\author{Aidi Li\thanks{School of Mathematical Sciences, Zhejiang University, Hangzhou, Zhejiang 310058, China
  (\email{22235031@zju.edu.cn}).}
\and Yuwen Li\thanks{Corresponding author. School of Mathematical Sciences, Zhejiang University, Hangzhou, Zhejiang 310058, China
  (\email{liyuwen@zju.edu.cn}).} 
}
\usepackage{amsopn}

\usepackage{booktabs} 
\usepackage{amsmath}
\usepackage[cal=cm, scr=rsfs]{mathalfa}
\usepackage{amssymb}
\usepackage{multirow} 
\usepackage{tabularx}

\ifpdf
\hypersetup{
  pdftitle={Rational Empirical Interpolation Methods with Applications},
  pdfauthor={Aidi Li and Yuwen Li}
}
\fi

\begin{document}
\maketitle
    
\begin{abstract}
We present a new rational approximation algorithm based on the empirical interpolation method for interpolating a family of parametrized functions to rational polynomials with invariant poles, leading to efficient numerical algorithms for space-fractional differential equations, parameter-robust preconditioning, and evaluation of  matrix functions.  Compared to classical rational approximation algorithms, the proposed method is more efficient for approximating a large number of target functions. In addition, we derive a convergence estimate of our rational approximation algorithm using the metric entropy numbers. Numerical experiments are included to demonstrate the effectiveness of the proposed method.
\end{abstract}

\begin{keyword}
rational approximation, empirical interpolation method, entropy numbers, fractional PDE, preconditioning, exponential integrator


\end{keyword}

\begin{MSCcodes}
41A20, 41A25, 65D15
\end{MSCcodes}

\section{Introduction}
In recent decades, physical modelling and numerical methods involving fractional order differential operators have received extensive attention, due to their applications in nonlocal diffusion processes (cf.~\cite{bucur2016nonlocal}). The fractional order operators could be defined using fractional Sobolev spaces or the spectral decomposition of integer order operators, see \cite{lischke2018fractional,d2020numerical}. Direct discretizations of those fractional operators inevitably lead to dense
differential matrices and costly implementations. There have been several efficient numerical methods in \cite{BonitoPasciak2015,Vabishchevich2015,harizanov2018optimal,hofreither2020unified,DanczulSchoberl2022} reducing the inverse of spectral fractional diffusion to a combination of inverses of second order elliptic operators. In particular, it suffices to approximate the power function $x^{-s}$ with $0<s<1$ by a rational function of the form
\begin{equation}\label{partialfraction}
    r_n(x)=\sum_{i=1}^{n} \frac{c_i}{x+b_i}
\end{equation}
to efficiently solve $\mathcal{A}^su=f$, where $\mathcal{A}$ is a Symmetric and Positive-Definite (SPD) local differential operator. 
In addition, rational approximation is also quite useful for parameter-robust preconditioning of finite-element discretized complex multi-physics systems (cf.~\cite{HolterKuchtaMardal2021,BudisaHu2023}) and efficient evaluation of exponential-type matrix functions (cf.~\cite{lopez2006analysis,DruskinKnizhnermanZaslavsky2009}) in exponential integrators for stiff dynamical systems. 

In the literature, rational approximation algorithms include the classical Remez algorithm in \cite{PetrushevPopov1987}, the BURA method in \cite{harizanov2018optimal}, the AAA algorithm in \cite{nakatsukasa2018aaa}, the BRASIL algorithm in \cite{Hofreither2021}, the Orthogonal {and Chebyshev Greedy Algorithms in \cite{LiZikatanovZuo2024,AdlerHuWangXue2024}, the Lawson's iteration in \cite{ZhangYangYangZhang2024},} etc. For example, the seminal and highly efficient AAA algorithm makes use of the barycentric representation $\big(\sum_{i=1}^n\frac{w_iF_i}{z-z_i}\big)/\big(\sum_{i=1}^n\frac{F_i}{z-z_i}\big)$ of rational functions. The aforementioned algorithms are all designed for the rational approximation of a single target function. {Recently, useful variants of AAA algorithms have also been developed for efficient simultaneous rational approximation of  a set of target functions, see, e.g., ~\cite{monzon2020multi,GoseaGuttel2021,monzon2023Iterative,Rodriguez2023pAAA}.}

In this paper, we present a {Rational Approximation Algorithm via the Empirical Interpolation Method (rEIM), which produces rational approximants in the form \eqref{partialfraction}.} The rEIM is a variant of the classical Empirical Interpolation Method (EIM), which was first introduced in \cite{barrault2004empirical} and further generalized in, e.g.,  \cite{maday2007general, ChaturantabutSorensen2010,EftangGreplPatera2013,maday2013generalized,NegriManzoniAmsallem2015,NguyenPeraire2023}.
The EIM is a greedy algorithm designed to approximate parametric functions, particularly useful in the context of model reduction of parametrized Partial Differential Equations (PDEs) and high-dimensional data analysis. Our rEIM adaptively selects basis functions from the set $$\mathcal{D}=\left\{\frac{\eta+b}{x+b}\right\}_{b \in (0,+\infty)}$$
and generates rational approximants $\{\Pi_nf_j\}_{1\leq j\leq J}$ interpolating a family of target functions $\{f_j\}_{1\leq j\leq J}$ at another adaptive set of interpolation points. Unlike classical EIMs, each target $f_j$ is not contained in $\mathcal{D}$.


{The applications considered in this work require the partial fraction decomposition \eqref{partialfraction} of rational polynomials approximating certain target functions. However, classical AAA-type rational approximation algorithms produce barycentric representation of rational approximants and the transformation from the barycentric form to \eqref{partialfraction} leads to loss of orders of accuracy. In contrast, the rEIM} directly outputs approximants of the form \eqref{partialfraction}, thereby avoiding the error arising from computing the poles $\{-b_i\}_{1\leq i\leq n}$ and residues $\{c_i\}_{1\leq i\leq n}$. The rEIM is designed to efficiently interpolate a family of functions $\{f_j\}_{1\leq j\leq J}$ to rational forms 
\begin{equation*}
(\Pi_nf_j)(x)=\sum_{i=1}^n\frac{c_{j,i}}{x+b_i},    
\end{equation*}
compared with existing algorithms in \cite{harizanov2018optimal,nakatsukasa2018aaa,Hofreither2021,LiZikatanovZuo2024} for approximating a single function. The rEIM would gain computational efficiency when the number $J$ of target functions is large. We remark that the poles $\{-b_i\}_{1\leq i\leq n}$ of $\Pi_nf_j$ are invariant for all $1\leq j\leq J$, a feature that saves the cost of adaptive step-size selection, solving parametrized problems, and approximating matrix functions, see section \ref{sec:application} for details. 

Following the framework in  \cite{Li2024CGA}, we also derive a sub-exponential convergence rate of the rEIM:
\begin{equation*}
 \left\|f-\Pi_{n-1}f\right\|_{L^\infty(I)}=\left(1+L_{n-1}\right)\left(\prod_{k=1}^{n-1}\left(1+L_k\right)\right)^{\frac{1}{n}}\|f\|_{\mathscr{L}_1(\mathcal{D})}O(\exp(-cn^\frac{1}{2})),    
\end{equation*}
where $\Pi_{n-1}f$ is the rEIM interpolation at the $(n-1)$-th iteration, $L_{n}$ is the Lebesgue constant of $\Pi_n$, $\|\bullet\|_{\mathscr{L}_1(\mathcal{D})}$ is the variation norm, and $c>0$ is an absolute constant. Our key ingredient is a careful analysis of the asymptotic decay rate of the entropy numbers $\varepsilon_n(B_1(\widetilde{\mathcal{D}}))$ for an analytically parametrized dictionary $\widetilde{\mathcal{D}}$ (see Theorem \ref{thmnwidth}), which is of independent interest in approximation and learning theory (cf.~\cite{Lorentz1996,SiegelXuFoCM,CohenDeVorePetrova2022}).
In \cite{DruskinKnizhnermanZaslavsky2009,DanczulHofreitherSchoberl2022},  rational Krylov space methods with a priori given invariant {Zolotarev} poles are used for computing matrix exponentials and solving fractional PDEs. Compared with \cite{DanczulHofreitherSchoberl2022}, the convergence analysis of the rEIM applies to a posteriori selected nested poles and our algorithm (Algorithm \ref{alg:rEIM}) directly computes the poles as well as interpolation points without using extra Gram--Schmidt orthogonalization as in the rational Krylov method.

Throughout this paper, $C$ is a positive generic constant that may change from line to line but independent of target functions and $n$. By $A\lesssim B$ we mean $A\leq CB$. 
The rest of this paper is organized as follows. In section \ref{sec:EIM}, we present the rEIM and its convergence estimate. In section \ref{sec:application}, we discuss important applications of the rEIM. section \ref{sec:convergenceREIM} is devoted to the convergence analysis of relevant entropy numbers and the rEIM. Numerical experiments are presented in section \ref{sec:num}.

\section{Rational Approximation via the Empirical Interpolation Method}\label{sec:EIM}
Throughout this paper, we shall focus on a positive interval $I$ with left endpoint $\eta>0$.
Given a parameter set $\mathcal{P}\subset\mathbb{R}^d$ and a collection of parameter-dependent functions $\widetilde{\mathcal{D}}=\{\tilde{g}(\bullet,\mu): \mu\in\mathcal{P}\}\subset C(I)$ on $I$, the classical EIM selects a set of functions $\tilde{g}(\bullet,\mu_1)$, ..., $\tilde{g}(\bullet,\mu_n)$ and interpolation points $\{x_1, \ldots, x_n\}\subset I$. Then for each parameter $\mu$ of interest, the EIM constructs $f_n(x,\mu) = \sum_{m=1}^n\beta_m(\mu)q_m(x)$ interpolating $f(\bullet,\mu)$ at $x_1, \ldots, x_n$, where  $q_1, ... , q_n$ are interpolation basis functions such that ${\rm Span}\{q_1, ..., q_n\}={\rm Span}\{\tilde{g}(\bullet,\mu_1), ..., \tilde{g}(\bullet,\mu_n)\}$.  

In order to efficiently {approximate} a family of parametrized functions by rational functions of the form \eqref{partialfraction}, we make use of the following rational dictionary 
\begin{equation}\label{rationaldictionary}
\mathcal{D}= \mathcal{D}((0,\infty))=\left\{g(\bullet,b)\in C(I): g(x,b)=\frac{\eta+b}{x+b},~b \in(0,\infty)\right\}
\end{equation}
and its subset $\mathcal{D}(\mathcal{B})\subset\mathcal{D}((0,\infty))$, where $\mathcal{B}\subset(0,\infty)$ is a problem-dependent and user-specified finite set. 
For the ease of subsequent analysis in section \ref{sec:convergenceREIM}, each $g$ in $\mathcal{D}$ is normalized such that $\|g\|_{L^\infty(I)}=1$.

In the context of classical EIMs, a linear combination of $g(\bullet,b_1), \ldots, g(\bullet,b_n)$ at parameter instances $b_1, \ldots, b_n$ is used to approximate $g(\bullet,b)$ for varying input parameters $b$. We remark that our goal is different from classical EIMs. In Algorithm \ref{alg:rEIM}, we present ``rEIM", a variant of the EIM using the rational dictionary \eqref{rationaldictionary}, aiming at efficiently rational approximation of a family of functions outside of $\mathcal{D}$.

\begin{algorithm}[thp]
\caption{{Rational Approximation via the EIM (rEIM)}}\label{alg:rEIM}
  \begin{algorithmic}
    \STATE \textbf{Input:} an integer $n>0$,  a dictionary $\mathcal{D}(\mathcal{B})$
and a set $\Sigma$ of possible interpolation points; set $\Pi_0=0$.

   \FOR{$m=1:n$}
   \STATE select $g_m=g(\bullet,b_m)\in\mathcal{D}(\mathcal{B})$ such that \begin{equation*}
    \|g_m-\Pi_{m-1}g_m\|_{L^\infty(I)}=\max_{g\in\mathcal{D}(\mathcal{B})}\|g-\Pi_{m-1}g\|_{L^\infty(I)};
\end{equation*}
   \STATE set $r_m=g_m-\Pi_{m-1}g_m$ and select $x_m\in\Sigma$ such that 
\begin{equation*}
|r_m(x_m)|=\max_{x\in\Sigma}|r_m(x)|;
\end{equation*}  
{set $\mathbb{G}_m=(g(x_i,b_j))_{1\leq i,j\leq m}$ and construct $\Pi_m$ as
}
\begin{equation}\label{Gm}
    \Pi_mf=(g(\bullet,b_1),\ldots,g(\bullet,b_m))\mathbb{G}_m^{-1}(f(x_1),\ldots,f(x_m))^\top;
\end{equation}
   \ENDFOR
\STATE \textbf{Output:} the rational interpolant $\Pi_nf$ for a family of target functions $f$ which are not necessarily in $\mathcal{D}(\mathcal{B})$.
  \end{algorithmic}  
\end{algorithm}

It is straightforward to check that $\Pi_nf$ is a rational function and $(\Pi_nf)(x_i)=f(x_i)$ for $i=1, 2, \ldots, n.$ For a family of target functions, the rEIM is able to efficiently interpolate them by applying a small-scale matrix $\mathbb{G}_n^{-1}$ to $(f(x_1),\ldots,f(x_n))^\top$ with $20\leq n\leq40$, {see the sub-exponential convergence rate of the rEIM in section \ref{subsecConvergenceEIM} and the numerical examples in section \ref{sec:num}.}

{At the $m$-step of Algorithm \ref{alg:rEIM}, set $q_m=r_m/r_m(x_m)$ and \begin{equation*}
\mathbb{Q}_m=(q_j(x_i))_{1\leq i,j\leq m}=\begin{pmatrix}q_1(x_1)&\cdots&q_m(x_1)\\
\vdots&~&\vdots\\
q_1(x_m)&\cdots&q_m(x_m)\end{pmatrix},
\end{equation*}
which is a lower triangular matrix with unit diagonal entries. Then each $q_m$ is a rational polynomial and an alternative expression of the interpolation $\Pi_m$ is  
\begin{equation}\label{Qm}
    \Pi_mf:=(q_1,\ldots,q_m)\mathbb{Q}_m^{-1}(f(x_1),\ldots,f(x_m))^\top.
\end{equation}
An additional transformation  $(q_1,\ldots,q_m)=(g(\bullet,b_1),\ldots,g(\bullet,b_m))\mathbb{T}_m$ is required for converting $\Pi_mf$ in \eqref{Qm} into a partial fraction $\sum_{j=1}^m\frac{c_j}{x+b_j}$, where $\mathbb{T}_m$ is a triangular matrix of order $m$. 
Although $\mathbb{G}_m=\left(g(x_i,b_j)\right)_{1\leq i,j\leq m}$ is a Cauchy matrix with possibly a large condition number, we note that the explicit  inversion formula (cf.~\cite{Heinig1992Inversion}) of a Cauchy matrix is helpful for reducing the rounding error. The advantage of \eqref{Gm} is that it directly outputs a rational polynomial of the desired form \eqref{partialfraction}.}    

When the number $J$ of input target functions $\{f_j\}_{1\leq j\leq J}$ is large, the rEIM could be more efficient than repeatedly calling a classical rational approximation algorithm designed for a single target function. {In particular, the total computational time of the rEIM is $T_{\rm setup}+J\cdot T_{\rm online}$, where  $T_{\rm setup}$ is the elapsed time for selecting the poles $-b_1, \ldots, -b_n$ as well as the interpolation points $x_1, \ldots, x_n$ in Algorithm \ref{alg:rEIM}, and $T_{\rm online}$ is typically a short time for implementing each $\Pi_nf_j$ (amounting to 1-2 matrix-vector multiplication of order $n$). Therefore, the computational complexity of the rEIM is dominated by $O(Jn^2)$ when $J\gg1$. }
For any input target function $f$, the rEIM interpolant $\Pi_nf$ has a fixed set of poles $-b_1, \ldots, -b_n$. This feature improves the efficiency of the rEIM-based numerical solvers, see section \ref{sec:application} for details.

{At the $m$-th step in Algorithm \ref{alg:rEIM}, the rEIM computes coefficients $\{c_i\}_{1\leq i\leq m}$ in $\Pi_mf=\sum_{i=1}^m\frac{c_i}{x+b_i}$ by interpolation conditions $(\Pi_mf)(x_i)=f(x_i)$,  $i=1, 2, \ldots, m$. An alternative way is determining $\{c_i\}_{1\leq i\leq m}$ by a least-squares problem:
\begin{equation*}
(c_1,\ldots,c_m)=\arg\min_{(\tilde{c}_1,\ldots,\tilde{c}_m)\in\mathbb{R}^m}\sum_{X\in\Sigma}\left|f(X)-\sum_{i=1}^m\frac{\tilde{c}_i}{X+b_i}\right|^2,
\end{equation*}
which is equivalent to $\min_{\mathbf{c}\in\mathbb{R}^m}\|\mathbb{A}_m\mathbf{c}-\mathbf{F}\|_{\ell^2}$, where  $\mathbb{A}_m\in\mathbb{R}^{K\times m}$, $\mathbf{F}=(f(X))_{X\in\Sigma}\in\mathbb{R}^K$, and $K=\#\Sigma\gg m$ is the number of sample points in $\Sigma$. As explained in \eqref{Gm}, the interpolatory approach could be implemented by solving a small-scale linear system of equations $\mathbb{G}_m\mathbf{c}=\mathbf{f}$ with $\mathbb{G}_m\in\mathbb{R}^{m\times m}$ and $\mathbf{f}=(f(x_1),\ldots,f(x_m))^\top\in\mathbb{R}^m$, which is cheaper than computing the least-squares solution of $\min_{\mathbf{c}\in\mathbb{R}^m}\|\mathbb{A}_m\mathbf{c}-\mathbf{F}\|_{\ell^2}$. The greedy selection of interpolation points plays a crucial role in the stabilization of the proposed rEIM as well as classical EIMs (cf.~\cite{barrault2004empirical,ChaturantabutSorensen2010}) for model reduction. }

{\begin{remark}
Recently, several variants of the AAA algorithm have been developed for simultaneously approximating a set $\{f_j\}_{1\leq j\leq J}$ of target functions, see, e.g., \cite{monzon2020multi,monzon2023Iterative,Rodriguez2023pAAA}. The output of those algorithms is another set $\{\tilde{f}_j\}_{1\leq j\leq J}$ of rational functions  in the barycentric form:
\begin{equation*}
\tilde{f}_j(z)=\frac{\sum_{i=1}^n\frac{w_{i}f_j(z_{i})}{z-z_{i}}}{\sum_{i=1}^n\frac{w_{i}}{z-z_{i}}}\approx f_j(z),
\end{equation*}
where $\{z_{i}\}_{1\leq i\leq n}$ are interpolation points and $\{w_{i}\}_{1\leq i\leq n}$ are corresponding weights. When deriving efficient numerical methods for PDEs by rational approximation algorithms, we need to rewrite $\tilde{f}_j$ as a partial fraction. However, converting $\tilde{f}_j$ into a partial fraction of the form \eqref{partialfraction} leads to loss of accuracy, see section \ref{subsec:xalpha} for numerical examples. 
\end{remark}}

\begin{remark}{Recently, the lightening method (cf.~\cite{Solving2019Trefethen,GopalTrefethen2019PNAS}) has been developed for computing highly accurate numerical solutions of Laplace and Helmholtz equations on a planar domain $\Omega\subset\mathbb{R}^2$ with low computational complexity. For example, the lightening method approximates the solution of a Laplace equation by $p+\sum_ir_i$, where $p$ is a polynomial and each $r_i$ is a partial fraction along the bisector of the $i$-th corner of $\Omega$. The location of poles of each $r_i$ is a priori set to satisfy an exponential  clustering distribution towards zero.}

{The approach in lightening methods  essentially approximates a fixed corner singularity $x^s$ ($s>0$) using ansatz $\sum_{j=0}^md_jx^j+\sum_{i=1}^n\frac{c_i}{x+b_i}$, while one of our main interests is to produce partial fraction approximants $\sum_{i=1}^n\frac{c_i}{x+b_i}$ for a family of  functions like $x^{-s}$. Unfortunately, we are not able to construct such rational approximations with high accuracy by simply a priori fixing poles $\{-b_i\}_{1\leq i\leq n}$ and determining $\{c_i\}_{1\leq i\leq n}$ by least-squares fitting. In addition, although the approach in \cite{Solving2019Trefethen,GopalTrefethen2019PNAS} is quite efficient for functions with a positive power-type singularity, the rEIM as well as other rational approximation algorithms are able to produce accurate rational approximants for more general functions without using much analytic information of the target function.}
\end{remark}

\subsection{Convergence Estimate of the REIM}\label{subsecConvergenceEIM}
Let $\widetilde{\mathcal{D}} \subset X$ be a bounded set of elements in a Banach
space. In particular, $X=L^\infty(I)$ in the analysis of the REIM.
The symmetric convex hull of $\widetilde{\mathcal{D}}$ is defined as
\begin{equation*}
B_1(\widetilde{\mathcal{D}})=\overline{\left\{\sum_{j=1}^m c_j g_j: m \in \mathbb{N},~g_j \in \widetilde{\mathcal{D}},~\sum_{i=1}^m\left|c_i\right| \leq 1\right\}}.
\end{equation*}
Using this set, the so-called variation norm (cf.~\cite{BarronCohenDahmenDeVore2008}) $\|\bullet\|_{\mathscr{L}_1(\widetilde{\mathcal{D}})}$ on $X$ is 
\begin{equation*}
\|f\|_{\mathscr{L}_1(\widetilde{\mathcal{D}})}=\inf \left\{c>0: f \in c B_1(\widetilde{\mathcal{D}})\right\},
\end{equation*}
and the subspace 
$\mathscr{L}_1(\widetilde{\mathcal{D}}):=\left\{f \in X:\|f\|_{\mathscr{L}_1(\widetilde{\mathcal{D}})}<\infty\right\}\subset X.$
The main convergence theorem of the proposed REIM is based on the entropy numbers (see  \cite{devore1993constructive}) of a set $F\subset X$:
\begin{equation*}
\varepsilon_n(F)=\varepsilon_n(F)_X=\inf \left\{\varepsilon>0: F\text{ is covered by } 2^n \text { balls of radius } \varepsilon\text{ in }X\right\} .\\
\end{equation*}

The sequence $\{\varepsilon_n(F)\}_{n\geq0}$ converges to 0 for any compact set $F$. We remark that classical literature {relates} the error of EIM-type algorithms  to the Kolmogorov $n$-width of $F$, see \cite{maday2016convergence}. Alternatively, we shall follow the  framework in \cite{LiSiegel2024,Li2024CGA} and derive an entropy-based convergence estimate of the REIM.
\begin{theorem}\label{thm:rEIMerror}
Let $L_n:=\sup _{0\neq g\in {\rm Span}\{\mathcal{D}\}} \frac{\left\|\Pi_n g\right\|_{{L^\infty(I)}}}{\|g\|_{{L^\infty(I)}}}$ and $S_n$ be the volume of the $n$-dimensional unit ball. For any $f\in \mathscr{L}_1(\mathcal{D})$, the rEIM (Algorithm \ref{alg:rEIM}) with $\mathcal{B}=(0,\infty)$ satisfies  
\begin{align*}
 &\left\|f-\Pi_{n-1}f\right\|_{L^\infty(I)}\\
 &\leq \left(1+L_{n-1}\right)\left(\prod_{k=1}^{n-1}\left(1+L_k\right)\right)^{\frac{1}{n}}\|f\|_{\mathscr{L}_1(\mathcal{D})}(nS_n)^\frac{1}{n}n \varepsilon_n(B_1(\mathcal{D}))_{L^\infty(I)}.    
\end{align*}
\end{theorem}
\begin{proof}
We start with a convergence estimate of the EIM developed in \cite{Li2024CGA}:
    \begin{equation}\label{EIMerror}
    \begin{aligned}
 &\sup _{g \in \mathcal{D}}\left\|g-\Pi_{n-1} g\right\|_{L^\infty(I)}\\
 &\leq\left(1+L_{n-1}\right)\left(\prod_{k=1}^{n-1}\left(1+L_k\right)\right)^{\frac{1}{n}}(nS_n)^\frac{1}{n}n \varepsilon_n(B_1(\mathcal{D}))_{L^\infty(I)}.    \end{aligned}
\end{equation}
By the definition of $\mathscr{L}_1(\mathcal{D})$, we can write each $f\in \mathscr{L}_1(\mathcal{D})$ as $$f=\sum_ic_ig_i,\quad\sum_i|c_i|\leq\|f\|_{\mathscr{L}_1(\mathcal{D})}$$ 
with each $g_i\in\mathcal{D}$. It then follows from \eqref{EIMerror} that
\begin{equation*}
\begin{aligned}
&\left\|f-\Pi_{n-1} f\right\|_{L^\infty(I)}\leq\sum_i |c_i|\left\|g_i-\Pi_{n-1} g_i\right\|_{L^\infty(I)}\\
  &\qquad\leq 
  \left(1+L_{n-1}\right)\left(\prod_{k=1}^{n-1}\left(1+L_k\right)\right)^{\frac{1}{n}}(nS_n)^\frac{1}{n}n \|f\|_{\mathscr{L}_1(\mathcal{D})}\varepsilon_n(B_1(\mathcal{D}))_{L^\infty(I)}. 
\end{aligned}
\end{equation*}
The proof is complete.
\end{proof}

Combining Theorem \ref{thm:rEIMerror} with the order of convergence of $\varepsilon_n(B_1(\mathcal{D}))_{L^\infty(I)}$ in Corollary \ref{thmentropy} yields the following  convergence rate of the rEIM.
\begin{corollary}\label{REIMconvergencerate}
For any $f\in \mathscr{L}_1(\mathcal{D})$, there exists an absolute constant $\beta>0$ independent of $f$ and $n$ such that the rEIM with $\mathcal{B}=(0,\infty)$ satisfies
\begin{equation*}
 \left\|f-\Pi_{n-1}f\right\|_{L^\infty(I)}\lesssim\left(1+L_{n-1}\right)\left(\prod_{k=1}^{n-1}\left(1+L_k\right)\right)^{\frac{1}{n}}\|f\|_{\mathscr{L}_1(\mathcal{D})}\exp(-\beta n^\frac{1}{2}).    
\end{equation*}
\end{corollary}
In section \ref{sec:convergenceREIM}, we shall discuss the membership of special target functions such as $(x^{s}+k)^{-1}$ in  $\mathscr{L}_1(\mathcal{D})$ and the entropy numbers of $B_1(\mathcal{D})_{L^\infty(I)}$. In the worst-case scenario, the Lebesgue constant $L_{n}$ {could grow exponentially}, see \cite{maday2007general}. However, it is widely recognized that such a pessimistic phenomenon will not happen in practical applications (see \cite{barrault2004empirical}). We shall test the growth of $L_n$ in section \ref{subsec:xalpha}.

\subsection{Rational Orthogonal Greedy Algorithm}

Another dictionary-based rational approximation method is the Rational Orthogonal Greedy Algorithm (ROGA) based on a rational dictionary developed in \cite{LiZikatanovZuo2024}, a variant of the classical OGA (cf.~\cite{DeVoreTemlyakov1996}). That algorithm constructs a sparse $n$-term rational approximation $f_n =  \sum_{i=1}^{n}c_ig_i$ for $f\in L^2(I)$ based on the dictionary \eqref{rationaldictionary}, see Algorithm \ref{alg:ROGA}.

\begin{algorithm}[H]
  \caption{Rational Orthogonal Greedy Algorithm}\label{alg:ROGA}
  \begin{algorithmic}
    \STATE \textbf{Input:} an integer $n>0$,  a rational dictionary $\mathcal{D}(\mathcal{B})$ in $L^2(I)$; set $f_0=0$.

   \FOR{$m=1:n$}
\STATE  compute $$g_m=\arg \max _{g \in \mathcal{D}(\mathcal{B})}\big|\left\langle g, f-f_{m-1}\right\rangle_{L^2(I)}\big|;$$
\STATE compute the $L^2$ orthogonal projection $f_m$ of 
$f$ onto ${\rm Span}\{g_1,g_2,...,g_m\}$;
   \ENDFOR
  \end{algorithmic}  
\end{algorithm}

Convergence of the OGA has been investigated in e.g., \cite{DeVoreTemlyakov1996,Temlyakov2018,LiSiegel2024}. In particular, \cite{LiSiegel2024} derives a sharp convergence estimate of the OGA based on the entropy numbers: 
\begin{equation*}\label{convergerateOGA}
\left\|f-f_n\right\|_{L^2(I)} \leq \frac{\left(n ! S_n\right)^{\frac{1}{n}}}{\sqrt{n}}\|f\|_{\mathscr{L}_1(\mathcal{D})} \varepsilon_n(B_1(\mathcal{D}))_{L^2(I)}.
\end{equation*}   
It then follows from the above estimate and the bound of $\varepsilon_n(B_1(\mathcal{D}))_{L^2(I)}$ in Corollary  \ref{thmentropy}  that the rational OGA is exponentially convergent.
\begin{corollary}\label{ROGAconvergencerate}
For the rational OGA (Algorithm \ref{alg:ROGA}) with $\mathcal{B}=(0,\infty)$ and $f\in\mathscr{L}_1(\mathcal{D})$, there exists an absolute constant $\gamma>0$ independent of $f$ and $n$ such that
\begin{equation*}
 \left\|f-f_n\right\|_{L^2(I)}\lesssim\|f\|_{\mathscr{L}_1(\mathcal{D})}\exp(-\gamma n^\frac{1}{2}).    
\end{equation*}
\end{corollary}

\section{Applications of the rEIM}\label{sec:application}
\subsection{Fractional Order PDEs}\label{subsec:Fractional}
Under the homogeneous Dirichlet boundary condition, a fractional order PDE of order $s \in(0,1)$ is 
\begin{equation}\label{problemFL}
\begin{aligned}
    \mathcal{A}^s u&=f\quad\text{ in }\Omega,\\
    u&=0\quad\text{ on }\partial\Omega,
\end{aligned}
\end{equation}
where $\mathcal{A}$ is a SPD compact operator. Let $0<\lambda_1\leq\lambda_2\leq\lambda_3\leq\cdots$ be the eigenvalues of $\mathcal{A}$ and $u_1$, $u_2$, $u_3$, $...$ the associated orthonormal eigenfunctions under the $L^2(\Omega)$-inner product $(\bullet,\bullet)$. The spectral {fractional} power of $\mathcal{A}$ is defined by
\begin{equation}\label{definefraction}
 \mathcal{A}^su = \sum_{i=1}^{\infty} \lambda_i^s(u,u_i) u_i. 
\end{equation}
When $\mathcal{A}=-\Delta$ is the negative Laplacian, \eqref{problemFL} reduces to the spectral fractional Poisson equation. For $\mathcal{A}u=-\nabla\cdot(a\nabla u)+cu$, $\mathcal{A}^s$ describes a fractional diffusion process. 

Direct discretizations such as the Finite Difference Method (FDM) and Finite Element Method (FEM) for \eqref{problemFL} lead to  dense linear systems. To remedy this situation, quadrature formulas and rational approximation algorithms are introduced in, e.g., \cite{harizanov2018optimal,harizanov2019comparison,hofreither2020unified} to approximate the solution $u$ of \eqref{problemFL} using a linear combination of numerical solutions of several shifted integer-order problems. Let $\mathcal{A}_h: \mathcal{V}_h\rightarrow\mathcal{V}_h$ be a discretization of $\mathcal{A}$ with maximum and minimum eigenvalues $\lambda_{\max}=\lambda_{\max}(\mathcal{A}_h)$ and $\lambda_{\min}=\lambda_{\min}(\mathcal{A}_h)$. Assume the output $r_n$ of the rEIM  is an accurate approximation of $x^{-s}$ over $[\lambda_{\min},\lambda_{\max}]$:
\begin{equation}\label{rn}
    r_n(x) = \sum_{i=1}^{n} \frac{c_i}{x+b_i}\approx x^{-s},\quad \lambda_{\min}\leq x\leq\lambda_{\max}.
\end{equation}
Let $\mathcal{I}$ be the identity operator. By observing $u=\mathcal{A}^{-s}f\approx\sum_{i=1}^nc_i(\mathcal{A}+b_i\mathcal{I})^{-1}f$ and using $r_n$ in \eqref{rn}, 
\begin{equation}\label{rationalPDEsolution}
    u_h:=\sum_{i=1}^nc_i(\mathcal{A}_h+b_i\mathcal{I}_h)^{-1}f_h
\end{equation}
is a numerical solution of \eqref{problemFL} with $\mathcal{I}_h: \mathcal{V}_h\rightarrow\mathcal{V}_h$ being the identity operator on the discrete level. The error $\|u-u_h\|_{L^2(\Omega)}$ is determined by the accuracy of rational approximation $\max_{x\in[\lambda_{\min},\lambda_{\max}]}|x^{-s}-r_n(x)|$ (see \cite{hofreither2020unified}). The evaluation of $u_h$ is equivalent to solving a series of SPD integer-order elliptic problems \begin{equation}\label{integerproblem}
(\mathcal{A}_h+b_i\mathcal{I}_h)u_h^i=f_h,\quad i=1, \ldots, n.  
\end{equation}

When using the FDM, $\mathcal{A}_h$ is the finite difference matrix, $f_h$ is a vector recording the values of $f$ at grid points, and $\mathcal{V}_h=\mathbb{R}^N$. In the setting of FEMs, $f_h$ is the $L^2$ projection of $f$ onto a finite element subspace $\mathcal{V}_h\subset H_0^1(\Omega)$ and $\mathcal{A}_h$ is represented by the matrix $\mathbb{M}^{-1}\mathbb{A}$, where $\mathbb{A}$ and $\mathbb{M}$ are the FEM stiffness and mass matrices, respectively. With a basis $\{\phi_i\}_{1\leq i\leq N}$ of $\mathcal{V}_h$, the solution \eqref{rationalPDEsolution} is $u_h=(\phi_1,\ldots,\phi_N)\mathbf{u}$ with  \begin{equation}\label{problemdiscrete}
\mathbf{u}=\sum_{i=1}^nc_i(\mathbb{A}+b_i\mathbb{M})^{-1}\mathbf{f},
\end{equation}
where $\mathbb{A}=(\nabla\phi_j,\nabla\phi_i)_{1\leq i,j\leq N}$,  $\mathbb{M}=(\phi_j,\phi_i)_{1\leq i,j\leq N}$, and $\mathbf{f}=((f,\phi_1),\ldots,(f,\phi_N))^\top$. It is noted that $\{b_i\}_{1\leq i\leq n}$ remains the same for different values of fractional order $s$. As a consequence, it suffices to pre-compute solvers, e.g., multi-frontal LU factorization, multigrid prolongations, for each $\mathbb{A}+b_i\mathbb{M}$ one time to efficiently solve $(-\Delta)^s u=f$ with a large number of input fractional order $s$.

\subsection{Evolution Fractional PDEs}\label{subsec:parabolic}
Another natural application of the rEIM is numerically solving the space-fractional parabolic equation 
\begin{equation}\label{heatFL}
\begin{aligned}
    u_t+\mathcal{A}^s u&=f, \quad\text{ on }(0,T]\times\Omega,\\
 u(0,\bullet)&=u_0\quad\text{ on }\Omega
\end{aligned}
\end{equation}
under the homogeneous Dirichlet boundary condition. Given the temporal grid $0=t_0<t_1<\cdots<t_M=T$ with $t_m=m\tau, \tau=T/M$, the semi-discrete scheme for \eqref{heatFL} based on the implicit Euler method reads
\begin{equation*}
 \left(\frac{1}{\tau} \mathcal{I}+\mathcal{A}^s\right)u^m=\frac{1}{\tau}u^{m-1}+f^m, \quad m=1, \ldots, M,   
\end{equation*}
where $u^m$ is an approximation of $u(t_m,\bullet)$. 
Following the same idea in section \ref{subsec:Fractional}, we use the rEIM to construct a rational function $\sum_{i=1}^n\frac{c_i}{x+b_i}$ approximating $(x^{s}+1/\tau)^{-1}$ over $[\lambda_{\min},\lambda_{\max}]$. The resulting fully discrete scheme is
\begin{align*}
u_h^m=\sum_{i=1}^nc_i(\mathcal{A}_h+b_i\mathcal{I}_h)^{-1}\left(\frac{1}{\tau}u_h^{m-1}+f_h^m\right), \quad m=1, \ldots, M,
\end{align*}
where $u_h^m\approx u(t_i,\bullet)$ and the meanings of $\mathcal{A}_h\approx \mathcal{A}$  and $f_h^m\approx f(t_i,\bullet)$ are explained in section \ref{subsec:Fractional}.

\subsection{Adaptive Step-Size Control}\label{subsec:adaptive}
Next we consider a non-uniform temporal grid $0=t_0<t_1<\cdots<t_N=T$ with $t_m=t_{m-1}+\tau_m$ and variable step-size control of each $\tau_m$. The numerical scheme for \eqref{heatFL} is as follows:
\begin{equation}\label{Euler}
u_h^m=\sum_{i=1}^nc_{i,m}(\mathcal{A}_h+b_i\mathcal{I}_h)^{-1}\left(\frac{1}{\tau_m}u_h^{m-1}+f_h^m\right), \quad m=1, 2, \ldots,
\end{equation}
where $\sum_{i=1}^n\frac{c_{i,m}}{x+b_i}$ is the rEIM interpolant of $(x^{s}+1/\tau_m)^{-1}$. We remark that $\tau_m$ might change at each time-level and the rEIM is {suitable} for generating rational approximants for a large number of $\tau_m$. Moreover, $b_1, \ldots, b_n$ in the rEIM remain the same for different values of $\tau_m$. This feature enables efficient computation of $u_h^m$ by inverting a fixed set of operators $\mathcal{A}_h+b_i\mathcal{I}_h$ independent of the variable step-size, saving computational cost for the same reason explained in section \ref{subsec:Fractional}. {Due to the invariance of $\{b_i\}_{1\leq i\leq n}$, a series of time-independent solvers $\{\mathcal{B}_i\}_{1\leq i\leq n}$ with $\mathcal{B}_i\approx(\mathcal{A}_h+b_i\mathcal{I}_h)^{-1}$ could be set up at the initial stage and reused at subsequent time levels for any step size $\tau_m$. As the time level $m$ varies, it suffices to apply $\{\mathcal{B}_i\}_{1\leq i\leq n}$ to the right side vector $\frac{1}{\tau_m}u_h^{m-1}+f_h^m$ in \eqref{Euler}. The coefficients $\{c_{i,m}\}_{1\leq i\leq n}$ are obtained by interpolating $(x^s+1/\tau_m)^{-1}$ using the rEIM only with cost $O(n^2)$, which is crucial for saving the cost of adaptive step size control based on solving \eqref{Euler} with many tentative step sizes $\tau_m$.} In particular, the advantage of invariant $\{b_i\}_{i=1}^n$ is also true for  integer-order parabolic equations. 







We use the numerical solution $\tilde{u}_h^m$ computed by a higher-order method, e.g., the BDF2 method as the reference solution and use $err_m:=\|u_h^m-\tilde{u}_h^m\|_{L^2(\Omega)}$ to predict the local error of \eqref{Euler} at each time level $t_{m+1}$ and adjust the step size $\tau_m$, see section \ref{expfractionalheat} for implementation details. 
The variable step-size BDF2  (cf.~\cite{JANNELLI2006Adaptive}) makes use of the backward finite difference formula $v^\prime(t_{m+1})\approx\kappa_{1,m}v(t_{m+1})+\kappa_{0,m}v(t_m)+\kappa_{-1,m}v(t_{m-1})$ with 
\begin{equation*}
  \kappa_{1,m}=\frac{2 \tau_m+\tau_{m-1}}{\tau_m\left(\tau_m+\tau_{m-1}\right)},\quad\kappa_{0,m}=-\frac{\tau_m+\tau_{m-1}}{\tau_{m-1} \tau_m},\quad\kappa_{-1,m}=\frac{\tau_m}{\tau_{m-1}\left(\tau_m+\tau_{m-1}\right)}.  
\end{equation*}
The fully discrete reference solution is computed as follows:
\begin{equation*}
\tilde{u}_h^{m+1}=\sum_{i=1}^n\tilde{c}_{i,m}(\mathcal{A}_h+b_i\mathcal{I}_h)^{-1}\left( f_h^{m+1}-\kappa_{0,m}u^m-\kappa_{-1,m}u^{m-1}\right), \quad m=1, 2, \ldots,
\end{equation*}
where $\sum_{i=1}^n\frac{\tilde{c}_{i,m}}{x+b_i}$ is the rEIM interpolant of $(x^s+\kappa_{1,m})^{-1}$. Then the rEIM for the family $\{(x^s+1/\tau_m)^{-1}\}_{m\geq1}$ used in \eqref{Euler} is also able to simultaneously generate rational approximation of $(x^s+\kappa_{1,m})^{-1}$ with little extra effort. The coefficients $\{b_i\}_{i=1}^n$ are the same as in \eqref{Euler}. 

\subsection{Preconditioning}\label{subsec:prec}
Recently, it has been shown in \cite{HolterKuchtaMardal2021,BoonHornkjolKuchtaMardalRuiz2022,HarizanovLirkovMargenov2022,BudisaHu2023} that fractional order operators are crucial in the design of parameter-robust preconditioners for complex multi-physics systems. For example, when solving a discretized Darcy--Stokes interface problem, the following theoretical block diagonal preconditioner (cf.~\cite{BudisaHu2023})
\begin{equation*}
\mathcal{B}_h:=\operatorname{diag}\left(\mathcal{A}_S, \mu K^{-1}(\mathcal{I}_h-\nabla_h\nabla_h\cdot), \mu^{-1}\mathcal{I}_h, K \mu^{-1} \mathcal{I}_h, \mathcal{S}_h\right)^{-1} 
\end{equation*}
is an efficient solver robust with respect to the viscosity $\mu>0$ and the permeability $K>0$, see \cite{BudisaHu2023} for details. The {$(5,5)$}-block of $\mathcal{B}_h$ is $$\mathcal{S}_h:=\mu^{-1}\mathcal{A}_{\Gamma,h}^{-1 /2}+K \mu^{-1}\mathcal{A}_{\Gamma,h}^{1 / 2},$$
where $\mathcal{A}_{\Gamma,h}$ is a discretization of the shifted Laplacian $-\Delta_\Gamma+\mathcal{I}_\Gamma$ on the Darcy--Stokes interface $\Gamma$. Over the interval $[\lambda_{\min}(\mathcal{A}_{\Gamma,h}),\lambda_{\max}(\mathcal{A}_{\Gamma,h})]$, applying the rEIM to $f(x)=(\mu^{-1} x^{-\frac{1}{2}}+K\mu^{-1} x^\frac{1}{2})^{-1}$ yields a rational interpolant $\sum_{i=1}^n\frac{c_i}{x+b_i}$ to $f$ and a spectrally equivalent operator
\begin{align*}
    \sum_{i=1}^nc_i(\mathcal{A}_{\Gamma,h}+b_i\mathcal{I}_{\Gamma,h})^{-1},
\end{align*}
serving as a  {$(5,5)$}-block in the practical form of the parameter-robust preconditioner. The rEIM-based preconditioner is {particularly} suitable for solving a series of multi-physics systems because $\{b_i\}_{i=1}^n$ are independent of the physical parameters $\mu, K$ and an approximate inverse, e.g., algeraic multigrid, of $\mathcal{A}_{\Gamma,h}+b_i\mathcal{I}_{\Gamma,h}$ could be re-used for different values of $\mu$ and $K$. 

\subsection{Approximation of Matrix Exponentials}\label{subsec:exp} The last example is the stiff or highly oscillotary system of ordinary differential equations 
\begin{equation}\label{ODE}
\mathbf{u}^\prime+\mathbb{L}\mathbf{u}=\mathbf{f}(\mathbf{u})    
\end{equation}
with $\mathbb{L}\in\mathbb{R}^{N\times N}$ and $\mathbb{L}\mathbf{u}$ being a dominating linear term. When numerically solving \eqref{ODE},
exponential integrators often exhibit superior stability and accuracy (cf.~\cite{HochbruckOstermann2010,LiWu2016SISC}). For example, the simplest exponential integrator for \eqref{ODE} is the following exponential Euler method:  
\begin{align*}
    \mathbf{u}_{m+1}=\exp(-\tau_m \mathbb{L})\mathbf{u}_m+\varphi(-\tau_m\mathbb{L})\mathbf{f}(\mathbf{u}_m),
\end{align*}
where $\tau_m$ is the step size at time $t_m$, $\varphi(x)=(\exp(x)-1)/x$ and $\mathbf{u}(t_m)\approx\mathbf{u}_m$.
Interested readers are referred to  \cite{HochbruckOstermann2010} for more advanced  exponential integrators. In practice, $\mathbb{L}$ is often a large and sparse symmetric positive semi-definite matrix, e.g., when \eqref{ODE} arises from semi-discretization of PDEs, and it is impossible to directly compute $\exp(-\tau_m \mathbb{L})$, $\varphi(-\tau_m\mathbb{L})$. In this case, iterative methods (cf.~\cite{HochbruckLubich1997,Higham2008}) are employed to approximate the matrix-vector products $\exp(-\tau_m\mathbb{L})\mathbf{v}$, $\varphi(-\tau_m\mathbb{L})\mathbf{v}$. An alternative way is to interpolate $\exp(-\tau_mx)$ and $\varphi(-\tau_mx)$ by the rEIM on $[\lambda_{\min}(\mathbb{L}),\lambda_{\max}(\mathbb{L})]$:
\begin{equation}\label{expphiapprox}
\begin{aligned}
        \exp(-\tau_mx) \approx r_n(x)&= \sum_{i=1}^n\frac{c_i}{x+b_i},\\
        \varphi(-\tau_mx) \approx \tilde{r}_n(x)&= \sum_{i=1}^n\frac{\tilde{c}_i}{x+b_i}.
\end{aligned}
\end{equation}
Then the rEIM-based approximation of $\exp(-\tau_m\mathbb{L})\mathbf{v}$ and $\varphi(-\tau_m\mathbb{L})\mathbf{v}$ is 
\begin{align*}
    \exp(-\tau_m\mathbb{L})\mathbf{v}&\approx \sum_{i=1}^n c_i(\mathbb{L}+b_i\mathbb{I})^{-1}\mathbf{v},\\
    \varphi(-\tau_m\mathbb{L})\mathbf{v}&\approx \sum_{i=1}^n \tilde{c}_i(\mathbb{L}+b_i\mathbb{I})^{-1}\mathbf{v},
\end{align*}
where $\mathbb{I}$ is the identity matrix.
It is also necessary to evaluate extra matrix functions $\varphi_2(-\tau_m\mathbb{L})$, $\varphi_3(-\tau_m\mathbb{L}), \ldots$ in higher-order exponential integrators. As mentioned before, the rEIM ensures that rational approximants of those functions share the same set of poles $-b_1, \ldots, -b_n$, which implies that only a fixed series of matrix inverse action $\{(\mathbb{L}+b_i\mathbb{I})^{-1}\mathbf{v}\}_{1\leq i\leq n}$ are needed regardless of the number of matrix functions and the value of $\tau_m$.

\section{Convergence Analysis}\label{sec:convergenceREIM}
In this section we analyze special cases of the membership of $\mathscr{L}_1(\mathcal{D})$ and the decay rate of the entropy numbers of the dictionary $\mathcal{D}$  in \eqref{rationaldictionary} over $I = [\eta,1]$. 

\subsection{Variation Norm of Functions}
\begin{lemma}\label{lemmaL1}
Let $\mathcal{W}\subseteq\mathbb{R}^m$ be a domain and  $\widetilde{\mathcal{D}} = \{\tilde{g}(\bullet,\omega): \omega \in \mathcal{W}\}$ on some interval $\tilde{I}$. Assume that $\tilde{g}$ is uniformly continuous on $\tilde{I}\times\mathcal{W}$. If a function $f$ could be written as
\begin{equation}\label{thrL11}
f(x) = \int_\mathcal{W} h(\omega)\tilde{g}(x,\omega)d\omega,    
\end{equation}
where $h$ satisfies $\int_\mathcal{W} |h(\omega)|d\omega <\infty$, then $f \in \mathscr{L}_1(\widetilde{\mathcal{D}})$.
\end{lemma}
\begin{proof}
Let $m(W)$ {be} the measure of a set $W\subset\mathbb{R}^m$. For any $\epsilon>0$, there exists $\delta>0$ such that whenever a partition $\{W_i\}_{i\geq1}$ of $\mathcal{W}$ and $\omega_i \in W_i$ satisfies $\sup_im(W_i)<\delta$, we have for any $x \in \tilde{I}$,
\begin{equation}\label{segmentation}
\big|f(x)-\sum _{i\geq1} m(W_i)h(\omega_i)\tilde{g}(x,\omega_i)\big| < \epsilon.  
\end{equation}
Then we can take a sufficiently small $\delta_1>0$ such that \eqref{segmentation} holds and 
\begin{equation*}
\sum _{i=1} m(W_i)|h(\omega_i)|\leq2\int_\mathcal{W} |h(\omega)|d\omega:=M.
\end{equation*}
Therefore, $f\in\mathscr{L}_1(\widetilde{\mathcal{D}})$ with
$\|f\|_{\mathscr{L}_1(\widetilde{\mathcal{D}})}\leq M$.
\end{proof}

Then we show that the target functions used in section \ref{subsec:Fractional}--\ref{subsec:adaptive} are contained in $\mathscr{L}_1(\mathcal{D})$.
\begin{corollary}\label{thrL1xalpha}
Let $\mathcal{D}$ be defined in  \eqref{rationaldictionary}. Given any $s\in (0,1) $, we have $(x^s+k)^{-1} \in \mathscr{L}_1(\mathcal{D})$ for $k\geq0$.
\end{corollary}
\begin{proof} 
When $k\geq0$, $(x^s+k)^{-1}$ belongs to the class of Stieltjes functions after 
$x^{-s}$ (cf.~\cite{Hirsch1972Int}) and admits the following integral representation (cf.~\cite{Schwarz2005Stieltjes}) 
\begin{equation}
    \frac{1}{x^s+k} = \frac{\sin \pi s}{\pi}\int_0^{+\infty} \frac{t^s}{\left(t^s \cos \pi s+k\right)^2+\left(t^s \sin \pi s\right)^2} \cdot \frac{1}{x+t} dt,\quad x>0.
\end{equation}
{It is straightforward to see that}
\begin{equation*}
\frac{\sin \pi s}{\pi}\int_0^{+\infty} \frac{t^s}{\left(t^s \cos \pi s+k\right)^2+\left(t^s \sin \pi s\right)^2} \cdot \frac{1}{\eta+t} dt = \frac{1}{\eta^s+k} < +\infty.
\end{equation*}
Combining the above results with Lemma \ref{lemmaL1} completes the proof.
\end{proof}

\subsection{Convergence Rate of Entropy Numbers}\label{entropy}
First we summarize two simple properties of the entropy numbers in the next lemma, see Chapter 7 in \cite{Temlyakov2018} and section 15.7 in \cite{devore1993constructive}.
\begin{lemma}
Let $B_X$ be a unit ball in a $d$-dimensional Banach space, then   
\begin{equation}\label{unitballentropy}
\varepsilon_n(B_X)_X \leq 3\cdot 2^{-n/d}.    
\end{equation}
For any $A, B \subset X$ and $m,n \geq 0$,
\begin{equation}\label{lemmaLorentz}
\varepsilon_{m+n}(A+B)_X \leq \varepsilon_m(A)_X+\varepsilon_n(B)_X.  
\end{equation} 
\end{lemma}

To analyze the decay rate of $\varepsilon_n(B_1(\mathcal{D}))$, we also make use of the Kolmogorov $n$-width of a set $F\subset X$:
\begin{equation*}
d_n(F)_X:=\inf_{\dim(V)=n} \sup _{f \in F} \inf _{g \in V}\|f-g\|_X,  
 \end{equation*} 
which describes the best possible approximation error of $F$ by an $n$-dimensional subspace in $X$. 

\begin{lemma}\label{cupD}
Assume $A = \cup_{i=1}^kA_i$ is a subset in $X$, then 
\begin{equation*}
    d_{kn}(A)_X \leq \max_{i=1,\ldots,k}\{d_n(A_i)_X\}.
\end{equation*}
\end{lemma}
\begin{proof}
For any $\epsilon>0$, there exists $n$-dimensional spaces $V_i,~ i =1,\ldots,k$, such that 
\begin{equation*}
   d_n(A_i)_X \geq \sup _{f \in A_i} \inf _{g \in V_i}\|f-g\|_X -\epsilon, \quad i=1,\ldots,k. 
\end{equation*}
Set $V = V_1+\ldots+V_k$, thus dim$(V)\leq kn$. Then it could be seen that
\begin{equation*}
\begin{aligned}
d_{kn}(A)_X &\leq\sup _{f \in A} \inf _{g \in V}\|f-g\|_X \leq \max_{i=1,...,k}\{\sup _{f \in A_i} \inf _{g \in V_i}\|f-g\|_X \} \\
&\leq \max_{i=1,...,k}\{d_n(A_i)_X\} + \epsilon .   
\end{aligned}
\end{equation*}
Sending $\epsilon$ to zero completes the proof.
\end{proof}

The classical Carl's inequality reveals a connection between  asymptotic convergence rates of the Kolmogorov $n$-width and entropy numbers: $d_n(K)_X =O(n^{-\alpha})\Longrightarrow\varepsilon_n(K)_X=O(n^{-\alpha})$ in the polynomial-decay regime, see \cite{Carl1981,Lorentz1996}. 
In the next theorem, we derive a sub-exponential analogue of the Carl's inequality.
\begin{theorem}\label{widthentropy}
Let $K$ be a compact set in a Banach space $X$. Then we have
\begin{equation}\label{Carl}
d_n(K)_X \leq C_1{\rm e}^{-r n^\alpha}\Longrightarrow\varepsilon_n(K)_X \leq C_2{\rm e}^{-C_3n^{\frac{\alpha}{\alpha+1}}}, 
\end{equation}
where $\alpha>0$ and the constants $C_2, C_3>0$ only depend on the constants $C_1>0$ and $r>0$.
\end{theorem}
\begin{proof}
It suffices to prove \eqref{Carl} with $n$ replaced by $2^n$. By the definition of Kolmogorov $n$-width, for each $i=0,1,...,n$, there exists a $2^i$-dimensional subspace $V_i$, such that for any $f\in K$, there exists an approximant $l_i(f)\in V_i$ such that 
\begin{equation*}
\|f-l_i(f)\|_X\leq C_1{\rm e}^{-r2^{i\alpha}}.
\end{equation*}
For $i=0,1,...,n$, we define $V_{-1}=\{0\}$, $l_{-1}(f)=0$, and  $t_i(f):=l_i(f)-l_{i-1}(f)$. Then $l_n(f)=\sum_{i=0}^n t_i(f)$, and {$t_i(f)\in V_i+V_{i-1} =:T_i$}. For $i=1,2,\ldots,n$, one can see that
\begin{equation*}
\begin{aligned}
    \dim(T_i)&\leq 2^i+2^{i-1} = 3\cdot2^{i-1},\\
    \|t_i(f)\|_X&\leq C_1{\rm e}^{-r2^{i\alpha}} + C_1{\rm e}^{-r2^{(i-1)\alpha}} \leq 2C_1{\rm e}^{-r2^{(i-1)\alpha}}.
\end{aligned}
\end{equation*}
Thus all of the $t_i(f)$ are covered by the ball of radius $2C_1{\rm e}^{-r2^{(i-1)\alpha}}$ at the origin in $T_i$. We use $2^{m_i}$ balls {$B(y_1^i,\varepsilon_{m_i})$, $B(y_2^i,\varepsilon_{m_i})$, ... , $B(y_{2^{m_i}}^i,\varepsilon_{m_i})$} of radius $\varepsilon_{m_i}$ to cover it, where each $B(y_{2^{m_i}}^i,\varepsilon_{m_i})$ is centered at $y_{2^{m_i}}^i$. It then follows from \eqref{unitballentropy} that for $i=1,2,\ldots,n$,
\begin{equation*}
\begin{aligned}
\inf_{j=1,2,...,2^{m_i} }\|t_i(f)-y_j^i\|_X &\leq 
{\varepsilon_{m_i}\left(2C_1{\rm e}^{-r2^{(i-1)\alpha}}B_{T_i}\right)_{T_i}} \\
&\leq 6C_1{\rm e}^{-r2^{(i-1)\alpha}}\left(2^{-m_i / (3\cdot2^{i-1})}\right) .   \end{aligned}
\end{equation*}
On the other hand, $\dim(T_0)=1$. By $\|f-l_0(f)\|_X\leq C_1{\rm e}^{-r}$ and $\|f\|_X\leq d_0(K)_X\leq C_1$, therefore $\|t_0(f)\|_X\leq 2C_1$. We also use $2^{m_0}$ balls $y_1^0,y_2^0,...,y_{2^{m_0}}^0$ of radius $\varepsilon_{m_0}$ to cover it. Thus,
\begin{equation*}
  \inf_{j=1,2,...,2^{m_0} }\|t_0(f)-y_j^0\|_X \leq 6C_1\cdot 2^{-m_0} .
\end{equation*}
Let {$Y:=\{y_{j_0}^0+y_{j_1}^1+...+y_{j_n}^n: ~j_i = 1,2,\ldots2^{m_i},~i = 0,1,\ldots n\}$}. Thus the number of elements in $Y$ does not exceed $2^{(\sum_{i=0}^{n}m_i)}$. We then choose $m_i$ as:
\begin{equation}\label{chooseni}
\begin{aligned}
m_i= \begin{cases}\frac{1}{\log2}2^{\frac{\alpha n}{\alpha+1}}, & i=0,\\\frac{3r}{\log2}2^{\frac{\alpha n}{\alpha+1}+i-1}, & 1\leq i< \frac{n}{\alpha+1}+1,\\0, & \frac{n}{\alpha+1}+1 \leq i \leq n . \end{cases}  \end{aligned}
\end{equation}
In what follows,
\begin{equation}\label{sumni}
\sum_{i=0}^{n}m_i = \frac{1}{\log2}2^{\frac{\alpha n}{\alpha+1}} + \sum_{i=1}^{[n/(\alpha+1)]+1}\frac{3r}{\log2}2^{\frac{\alpha n}{\alpha+1}+i-1} \leq \frac{6r+1}{\log2}2^{n}.
\end{equation}
Obviously $2^{(i-1)\alpha} \geq (i-1)\alpha$ when $i\geq 1$, and there exists $\gamma >0$ such that
$$2^{(i-1)\alpha} - 2^{\frac{n\alpha}{\alpha+1}} \geq \gamma\left(i-1-\frac{n}{\alpha+1}\right)\alpha$$ when $ i \geq n/(\alpha+1)+1$. Thus by \eqref{chooseni},
\begin{equation*}
\begin{aligned}
r2^{(i-1)\alpha}+\log2\frac{m_i}{3\cdot2^{i-1}}\geq \begin{cases}r(2^{\frac{\alpha n}{\alpha+1}}+(i-1)\alpha), & 1\leq i< \frac{n}{\alpha+1}+1,\\r(2^{\frac{\alpha n}{\alpha+1}}+\gamma(i-1-\frac{n}{\alpha+1})\alpha), & \frac{n}{\alpha+1}+1 \leq i \leq n . \end{cases}    
\end{aligned}    
\end{equation*}

Then we approximate $f \in K$ by elements of $Y$:
\begin{equation}\label{erroroff}
\begin{aligned}
\inf_{y\in Y}\|f-y\|_X & \leq \|f-l_n(f)\|_X + \sum_{i=0}^{n}\inf_{j_i}\left\|t_i(f)- y_{j_i}^i\right\|_X \\
& \leq C_1{\rm e}^{-r2^{n\alpha}} + 6C_1{\rm e}^{-r2^{\frac{\alpha n}{\alpha+1}}}+6C_1\sum_{i=1}^{i<n/(\alpha+1)+1}{\rm e}^{-r(2^{\frac{\alpha n}{\alpha+1}}+(i-1)\alpha)} \\
& \quad + 6C_1\sum_{i\geq n/(\alpha+1)+1}^{n}{\rm e}^{-r(2^{\frac{\alpha n}{\alpha+1}}+\gamma(i-1-\frac{n}{\alpha+1})\alpha)} \\
& \leq C_2{\rm e}^{-r2^{\frac{\alpha n}{\alpha+1}}},
\end{aligned}
\end{equation}
where the constant $C_2$ only depends on $C_1$ and $r$.
Then by \eqref{sumni} and \eqref{erroroff} we have
\begin{equation*}
\varepsilon_{\frac{6r+1}{\log2}2^{n}}(K)_X\leq C_2{\rm e}^{-r2^{\frac{\alpha n}{\alpha+1}}},
\end{equation*}
which implies
{
\begin{equation*}
\varepsilon_n(K)_X \leq C_2{\rm e}^{-r(\frac{n\log2}{6r+1})^{\frac{\alpha}{\alpha+1}}} =: C_2{\rm e}^{-C_3n^{\frac{\alpha}{\alpha+1}}}.   
\end{equation*}}
The proof is complete.
\end{proof}

Now we are in a position to present sub-exponential convergence $d_n(\widetilde{\mathcal{D}})_X$ and $\varepsilon_n(B_1(\widetilde{\mathcal{D}}))_X$ for an analytically smooth dictionary $\widetilde{\mathcal{D}}$.
\begin{theorem}\label{thmnwidth}
Let $\mathcal{W}$ be a convex domain in $\mathbb{R}^m$ and  $\mathscr{G}$ be a parameterization mapping $\omega \in \mathcal{W}$ to a function $\tilde{g}(\bullet,\omega)$ on $\tilde{I}$. If $\mathscr{G}$ has an analytic continuation on an open neighborhood $\mathcal{U} \subset \mathbb{C}^m$ of $\mathcal{W}$, then for $\widetilde{\mathcal{D}}=\mathscr{G}(\mathcal{W})$ it holds that
\begin{subequations}\label{nwidthrate}
\begin{align}
d_n(\widetilde{\mathcal{D}})_{L^{\infty}(\tilde{I})}\lesssim {\rm e}^{-C_4n^{\frac{1}{m}}}, \label{nwidthrate:a}\\
\varepsilon_n(B_1(\widetilde{\mathcal{D}}))_{L^{\infty}(\tilde{I})}\lesssim {\rm e}^{-C_5n^{\frac{1}{m+1}}}, \label{nwidthrate:b}
\end{align}
\end{subequations}
where the constants $C_4>0$ and $C_5>0$ are independent of $n$.
\end{theorem}
\begin{proof}
For a multi-index $\mathbf{a} = (a_1,a_2,\ldots,a_m)$, we adopt the conventional notation $|\mathbf{a}| = a_1+a_2+\cdots+a_m$, and $\mathbf{a}!=a_1!a_2!\cdots a_m!$, $\mathbf{w}^{\mathbf{a}}= \omega_1^{a_1}\omega_2^{a_2}\cdots \omega_m^{a_m}$, where $\mathbf{w}=(\omega_1,\omega_2,\ldots,\omega_m)$ is a vector.

Let $\mathbf{w}_0$ be an element in $\mathcal{W}$, and $\mathbf{h} = (h_1,h_2,\ldots ,h_m)$ be a vector such that $\mathbf{w}_0+\mathbf{h} \in \mathcal{W}$.  
Since $\mathscr{G}$ is analytic in $\mathcal{U}$, by the multivariate Taylor expansion formula, for any positive integer $n$, there exists a $\theta \in (0,1)$ such that
\begin{equation*}
\mathscr{G}(\mathbf{w}_0+\mathbf{h}) = \sum_{k=0}^n\sum_{|\mathbf{a}|=k}\frac{\partial^{a_1+a_2+\cdots+a_m}\mathscr{G} }{\partial\omega_1^{a_1}\cdots\partial\omega_m^{a_m}}(\mathbf{w}_0)\frac{\mathbf{h}^{\mathbf{a}}}{\mathbf{a}!} + R_n,
\end{equation*}
where $R_n$ is the Lagrange Remainder
\begin{equation*}
R_n = \sum_{|\mathbf{a}|=n+1}\frac{\partial^{a_1+a_2+\cdots+a_m}\mathscr{G} }{\partial\omega_1^{a_1}\cdots\partial\omega_m^{a_m}}(\mathbf{w}_0+\theta\mathbf{h})\frac{\mathbf{h}^{\mathbf{a}}}{\mathbf{a}!}.
\end{equation*}

Then we define a space $V_n$ by
\begin{equation*}
    V_n := {\rm Span}\left\{\frac{\partial^{a_1+a_2+\cdots+a_m}\mathscr{G} }{\partial\omega_1^{a_1}\cdots\partial\omega_m^{a_m}}(\mathbf{w}_0):~|\mathbf{a}| \leq n\right\}.
\end{equation*}
The number of vectors $\mathbf{a}$ that satisfy $|\mathbf{a}| = n$ does not exceed $(n+1)^{m-1}$, and the number of vectors $\mathbf{a}$ that satisfy $|\mathbf{a}| \leq n$ does not exceed $(n+1)^m$, we will prove this in Lemma \ref{lemmanumber}. Thus, $\text{dim}(V_n)\leq (n+1)^m$.

On the other hand, assume $\Gamma \subset \mathcal{U}$ is a closed loop surrounding $\mathcal{W}$, and denote $$d = \sup_{\omega',\omega''\in \mathcal{W}}|\omega'-\omega''|,~ \delta = \inf_{\zeta \in \Gamma,\omega \in \mathcal{W}}|\zeta-\omega|.$$ By the Cauchy integral formula, for any $\mathbf{h}$ with $\mathbf{w}_0+\mathbf{h}\in \mathcal{W}$,
\begin{equation*}
\begin{aligned}
   &\inf_{g \in V_n}\|\mathscr{G}(\mathbf{w}_0+\mathbf{h})-g\|_{L^{\infty}(\tilde{I})} \leq \|R_n\|_{L^{\infty}(\tilde{I})} \\
   & =\sum_{|\mathbf{a}|=n+1}\left\|\frac{\mathbf{h}^{\mathbf{a}}}{\mathbf{a}!}\cdot
   \frac{\mathbf{a}!}{(2 \pi i)^{m}} \int_{\partial \Gamma} \frac{\mathscr{G}\left(\mathbf{w}_0+\theta\mathbf{h}\right)}{\prod_{i=1}^m(\zeta_i-\omega_i')^{a_i+1}} d \zeta_1 \wedge  \cdots \wedge  d \zeta_m \right\|_{L^{\infty}(\tilde{I})}\\
   & \leq \sum_{|\mathbf{a}|=n+1}\left|\frac{\mathbf{h}^{\mathbf{a}}}{\mathbf{a}!}\cdot\frac{\mathbf{a}!}{(2 \pi i)^{m}}\cdot \frac{\sup_{\omega \in \mathcal{W}}\|\mathscr{G}(\omega)\|_{L^{\infty}(\tilde{I})}m(\partial\Gamma)}{\delta^{|\mathbf{a}|+m}}\right|\\
   & \lesssim (n+2)^{m-1}\left(\frac{d}{\delta}\right)^n,
   \end{aligned}
\end{equation*}
where the constant $m(\partial\Gamma)$ is the measure of $\partial\Gamma$, and $\omega_i'$ is the $i$-th component of $\mathbf{w}_0+\theta\mathbf{h}$. 

We first assume that $d < \delta$, then there exists $c_0$ such that  $(n+2)^{m-1}(d/\delta)^n \lesssim {\rm e}^{-c_0n}$ and
\begin{equation*}
d_{(n+1)^m}(\widetilde{\mathcal{D}})_{L^{\infty}(\tilde{I})} \leq \sup_{\omega\in \mathcal{W}}\inf_{g\in V_n}\|\mathscr{G}(\omega)-g\|_{L^{\infty}(\tilde{I})} \lesssim {\rm e}^{-c_0n},
\end{equation*}
which further implies that
\begin{equation}\label{nwidth}
d_n(\widetilde{\mathcal{D}})_{L^{\infty}(\tilde{I})} \lesssim {\rm e}^{-c_0n^{\frac{1}{m}}}.
\end{equation}

If $d \geq \delta$, we divide $\mathcal{W}$ into several parts $\mathcal{W}_1,\ldots,\mathcal{W}_k$ such that the diameter of each part is less than $\delta$. Let $\mathcal{D}_i=\{\mathscr{G}(\mathcal{W}_i)\}$, then $\widetilde{\mathcal{D}} = \cup_{i=1}^k \mathcal{D}_i$. Then by our proof for \eqref{nwidth}, there exists $c_1,\ldots,c_k >0$ such that
\begin{equation*}
d_n\left(\mathcal{D}_i\right)_{L^{\infty}(\tilde{I})} \lesssim {\rm e}^{-c_in^{\frac{1}{m}}},~ i =1,\ldots,k.
\end{equation*}
It then follows from Lemma \ref{cupD} that
\begin{equation}\label{knwidthD}
d_{kn}(\widetilde{\mathcal{D}})_{L^{\infty}(\tilde{I})} \leq \max\{d_n\left(\mathcal{D}_i\right)_{L^{\infty}(I)},~i=1,\ldots,k\} \lesssim {\rm e}^{-c'n^{\frac{1}{m}}},
\end{equation}
where the constant $c'=\min\{c_i,~i=1,\ldots,k\}$ depends only on $\mathcal{W}$ and $\Gamma$, and therefore depends only on $\mathcal{W}$ and $\mathcal{U}$. Then \eqref{nwidthrate:a} is proved by \eqref{knwidthD}. 

On the other hand, combining $d_n(B_1(\mathcal{D}^{(1)}))_{L^{\infty}(I)}=d_n(\mathcal{D}^{(1)})_{L^{\infty}(I)}$ (see \cite{Lorentz1996}) and \eqref{nwidthrate:a} with Theorem \ref{widthentropy} completes the proof of \eqref{nwidthrate:b}.
\end{proof}

\begin{lemma}\label{lemmanumber}
For any $n\geq 0$, the number of vectors $\mathbf{a}$ that satisfy $|\mathbf{a}| = n$ does not exceed $(n+1)^{m-1}$, and the number of vectors $\mathbf{a}$ that satisfy $|\mathbf{a}| \leq n$ does not exceed $(n+1)^m$.  
\end{lemma}
\begin{proof}
For $m\geq 1$ and $n\geq 0$, denote the number of vectors $\mathbf{a}$ such that $a_1+\ldots+a_m = n$ by $N(m,n)$. It is clear that $N(m,0)=1$, $ N(1,n) = 1$, $N(m,n)\leq N(m,n+1).$
By fixing $a_m$, we have the recurrence formula:
\begin{equation*}
    N(m,n) = \sum_{i=0}^n N(m-1,i) \leq (n+1)N(m-1,n).
\end{equation*}
Thus, by induction, $N(m,n) \leq (n+1)^{m-1}$, which implies that the first statement holds. Besides, 
\begin{equation*}
    \sum_{i=0}^n N(m,i) \leq \sum_{i=0}^n (i+1)^{m-1} \leq (n+1)\cdot(n+1)^{m-1} \leq (n+1)^m,
\end{equation*}
which implies the second statement.
\end{proof}

Finally, combining Theorems \ref{widthentropy} and \ref{thmnwidth}, we obtain sub-exponential decay of  $\varepsilon_n(B_1(\mathcal{D}))_{L^\infty(I)}$ for the dictionary  \eqref{rationaldictionary} in the next corollary.
\begin{corollary}\label{thmentropy}
Let $\mathcal{D}$ be defined in  \eqref{rationaldictionary}. For any $1\leq p\leq \infty$ we have 
\begin{equation*}\label{entropyrate2}
\varepsilon_n\left(B_1(\mathcal{D})\right)_{L^p(I)} \lesssim {\rm e}^{-C_6n^{\frac{1}{2}}},
\end{equation*}
where the constant $C_6$ are independent of $n$.
\end{corollary}

\begin{proof}
It suffices to prove the theorem with $p=\infty$ because of the simple relation $\|f\|_{L^p(I)}\lesssim\|f\|_{L^\infty(I)}$. First
we analyze the part $\mathcal{D}^{(1)}=\mathcal{D}((0,1])=\left\{g(\bullet,b)\right\}_{b \in (0,1]}$ of $\mathcal{D}=\mathcal{D}((0,\infty))$. Let $B_{1+\frac{\eta}{2}}(1)$ be the {disk} centered at $1$ with a radius of $(1+\eta/2)$ on the complex plane, where $\eta$ is the left endpoint of $I$. 
Since $(0,1] \subset B_{1+\frac{\eta}{2}}(1)$ and $g(\bullet,b)$ is analytic in $B_{1+\frac{\eta}{2}}(1)$, by Theorem \ref{thmnwidth} we directly have
\begin{equation}\label{entropyconludeD(1)}
\varepsilon_{n}(B_1(\mathcal{D}^{(1)}))_{L^{\infty}(I)}  \lesssim {\rm e}^{-C^{(1)}n}.    
\end{equation}
where the constant $C^{(1)}>0$ {is independent of} $n$.

Next we consider another part $\mathcal{D}^{(2)}=\mathcal{D}([1,+\infty))$ of $\mathcal{D}$. By replacing $b$ with $1/\tilde{b}$, it is equivalent to the dictionary $$\mathcal{D}^*:=\left\{h(\bullet,\tilde{b}):h(x,\tilde{b})=\frac{\tilde{b}\eta+1}{\tilde{b}x+1}\right\}_{\tilde{b} \in (0,1]}.$$
Since $h(\bullet,\tilde{b})$ is analytic in $B_{\frac{3}{2}}(1)$, there also exists a constant $C^{(2)}>0$ such that
\begin{equation}\label{entropyconludeD(2)}
\varepsilon_{n}(B_1(\mathcal{D}^{(2)}))_{L^{\infty}(I)} =\varepsilon_{n}(B_1(\mathcal{D}^*))_{L^{\infty}(I)} \lesssim {\rm e}^{-C^{(2)}n^{\frac{1}{2}}}.  
\end{equation}
Since $B_1(\mathcal{D}) \subset B_1(\mathcal{D}^{(1)})+ B_1(\mathcal{D}^{(2)})$, by substituting \eqref{entropyconludeD(1)} and \eqref{entropyconludeD(2)} into \eqref{lemmaLorentz}, we obtain
\begin{equation*}
\begin{aligned}
\varepsilon_{n}(B_1(\mathcal{D}))_{L^{\infty}(I)}  &\leq \varepsilon_{n/2}(B_1(\mathcal{D}^{(1)}))_{L^{\infty}(I)} +\varepsilon_{n/2}(B_1(\mathcal{D}^{(2)}))_{L^{\infty}(I)}\lesssim {\rm e}^{-C_6n^{\frac{1}{2}}},  
\end{aligned}    
\end{equation*}
where $C_6 = \min\{C^{(1)}/\sqrt2,C^{(2)}/\sqrt2\}$ is independent of $n$.
\end{proof}

\section{Numerical Experiments}\label{sec:num}
In this section, we test the performance of the rEIM for solving fractional PDEs. Given an upper bound $\Lambda$ of $\lambda_{\max}(\mathcal{A}_h)$, we choose to replace $\mathcal{A}_h$ and $f_h$ with $\mathcal{A}_h/\Lambda$ and $f_h/\Lambda^s$ in section \ref{subsec:Fractional} without change the numerical solutions $u_h$. Correspondingly, the rEIM in section \ref{subsec:Fractional} is applied to the target functions $x^{-s}$ over the rescaled  interval $[\eta,1]$, where $0 < \eta \leq\lambda_{\min}(\mathcal{A}_h)/\Lambda$. 

For the evolution fractional PDE in section \ref{subsec:adaptive}, the rescaled problem is 
\begin{equation*}
u_h^m=\sum_{i=1}^nc_{i,m}\left(\frac{\mathcal{A}_h}{\Lambda}+b_iI_h\right)^{-1}\frac{1}{\Lambda^s}\left(\frac{1}{\tau_m}u_h^{m-1}+f_h^m\right), \quad m=1, 2, \ldots,
\end{equation*}
where $\{c_{i,m}\}_{1\leq i\leq n}$, $\{b_i\}_{1\leq i\leq n}$ are determined by the rEIM interpolant
$\sum_{i=1}^n\frac{c_{i,m}}{x+b_i}$ of $\frac{1}{x^s+1/(\tau_m\Lambda^s)}$ over $[\eta,1]$. {Here $c_{i,m}$ depends on $s$ while $b_i$ is invariant as $s$ changes.}

{In each experiment, we implement Algorithm \ref{alg:rEIM} based on an unnormalized dictionary $\mathcal{D}(\mathcal{B})=\left\{g(\bullet,b)\in C(I): g(x,b)=\frac{1}{x+b},~b\in\mathcal{B}\right\}$ with slight abuse of notation. For simplicity, we use MATLAB's backslash `\textbackslash
' to evaluated the action of $\mathbb{G}_m^{-1}$ on a vector $\mathbf{v}$ in $\mathbb{R}^m$. It is observed that the numerical error of $\mathbb{G}_m^{-1}\mathbf{v}$ by MATLAB's backslash is well below $10^{-9}$. As a result, the numerical accuracy of Algorithm \ref{alg:rEIM} is not significantly affected by the large condition number of $\mathbb{G}_m$.}

The numerical accuracy of the rEIM as well as other rational approximation algorithms depends on fine tuning of discretization parameters such as the choice of the finite dictionary $\mathcal{D}(\mathcal{B})\subset\mathcal{D}((0,\infty))$ and the candidate set $\Sigma\subset(0,\infty)$ of {interpolation} points in Algorithm \ref{alg:rEIM}. {Due to the singularity of target functions at the origin, we choose to sample the dictionary $\mathcal{D}(\mathcal{B})$ and the set $\Sigma$ in a nonuniform way, i.e., $\mathcal{D}(\mathcal{B})$ and $\Sigma$ are increasingly denser as $b\in\mathcal{B}$ and $x\in\Sigma$ getting closer to zero. 
In our codes,  $\mathcal{B}$ and $\Sigma$ are tuned to optimize the performance of the rEIM.  
Based on our numerical experience, we set the range of $\mathcal{B}$ such that $\|\mathcal{D}(\mathcal{B})\|_{L^{\infty}(I)}$ is no less than the $L^\infty$ norm of target functions.} Interested {readers are referred} to the repository \href{https://github.com/yuwenli925/REIM}{github.com/yuwenli925/REIM} for implementation details of the practical rational approximation algorithms under numerical investigation.

\begin{figure}[thp]
\begin{minipage}{1.0\linewidth}
    \centering
    \includegraphics[width=.3\linewidth]{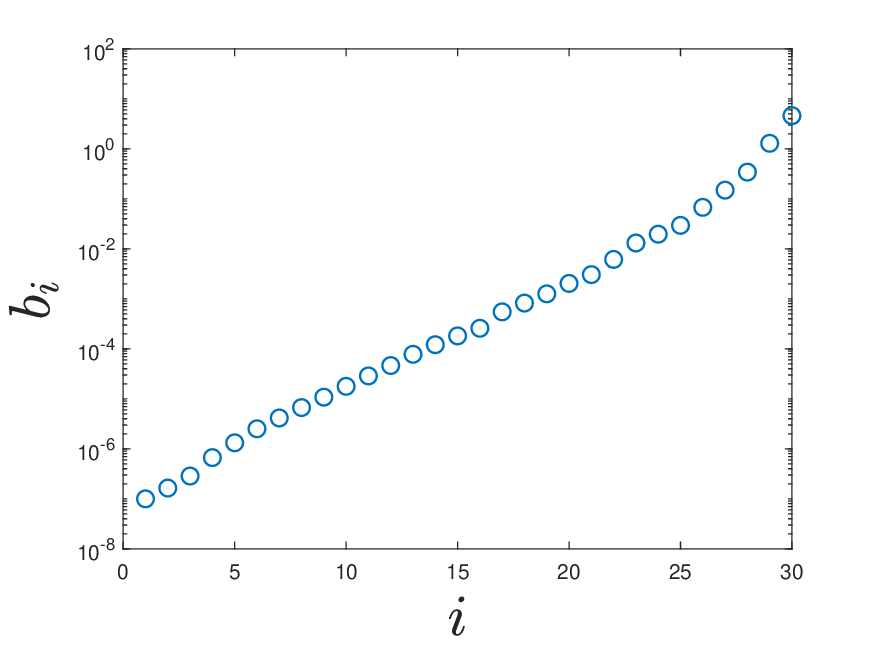}
    \includegraphics[width=.3\linewidth]{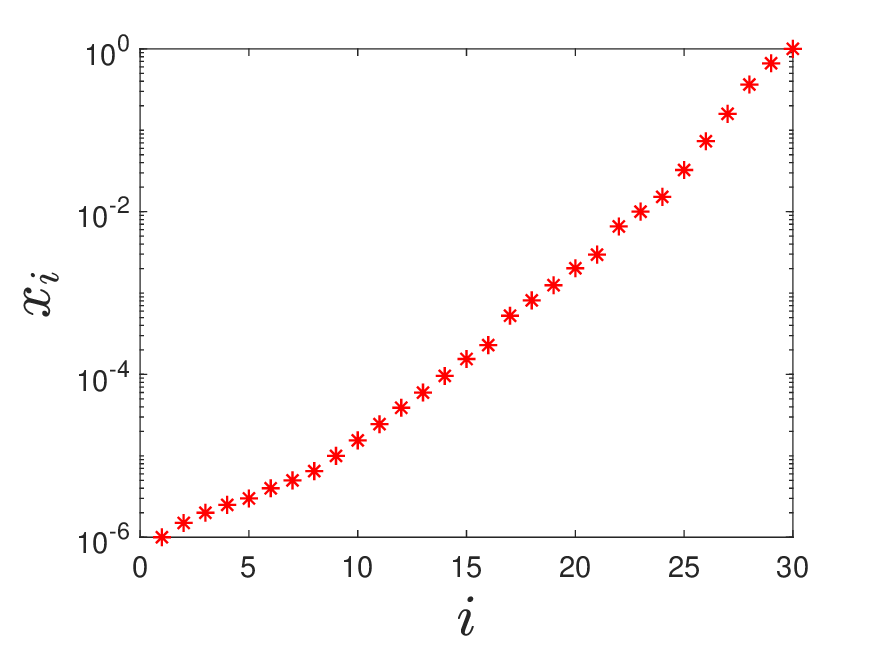} \includegraphics[width=.3\linewidth]{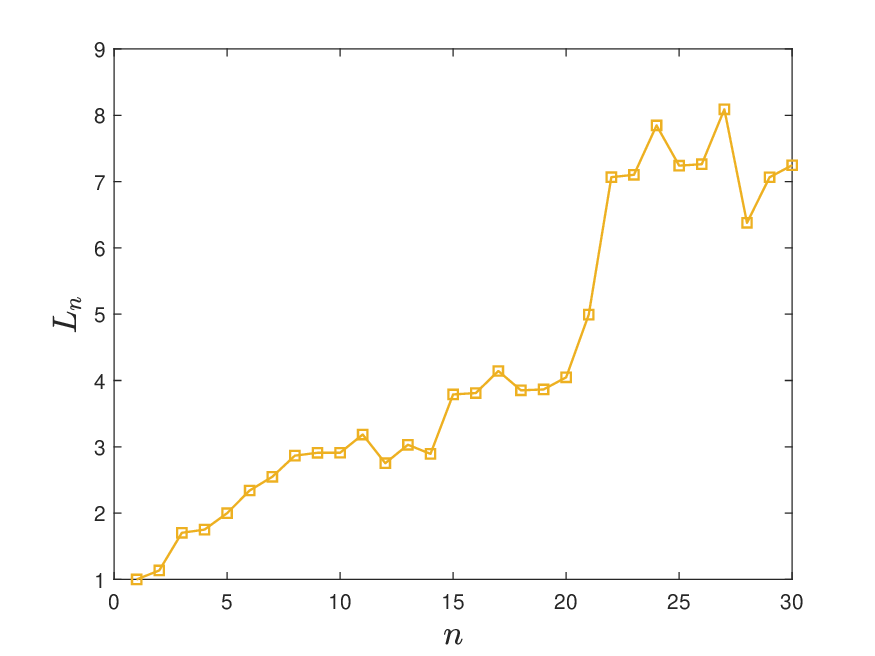}\vspace{8pt} 
    \caption{Opposite poles $b_i$ in the rEIM (left); rEIM interpolation points $x_i$ (middle); Lebesgue constant $L_n$ (right).}
    \label{fig:polesort}
\end{minipage}
\end{figure} 

\begin{figure}[thp]
\begin{minipage}{1.0\linewidth}
    \centering
    \includegraphics[width=.4\linewidth]{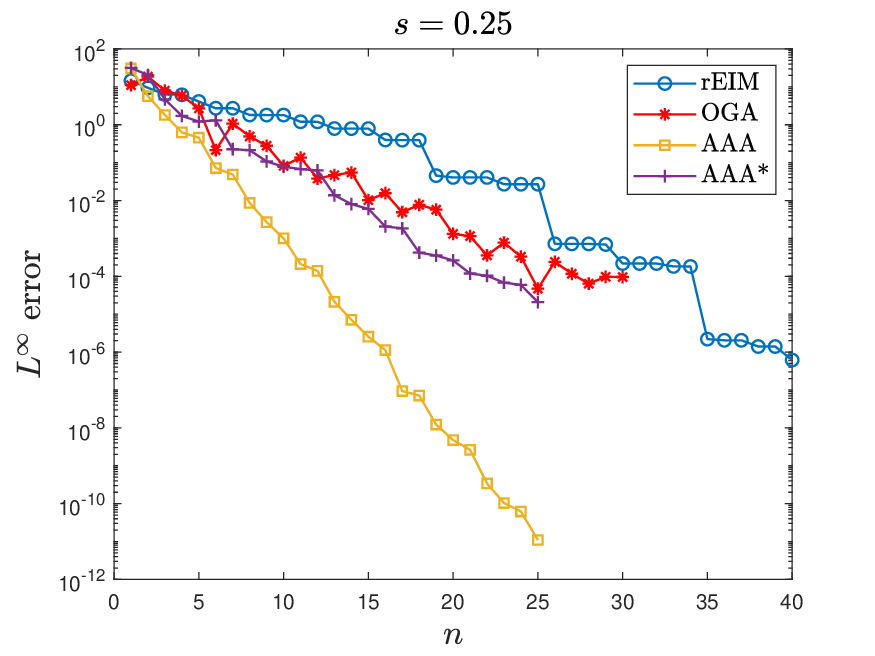}\hspace{8pt}
    \includegraphics[width=.4\linewidth]{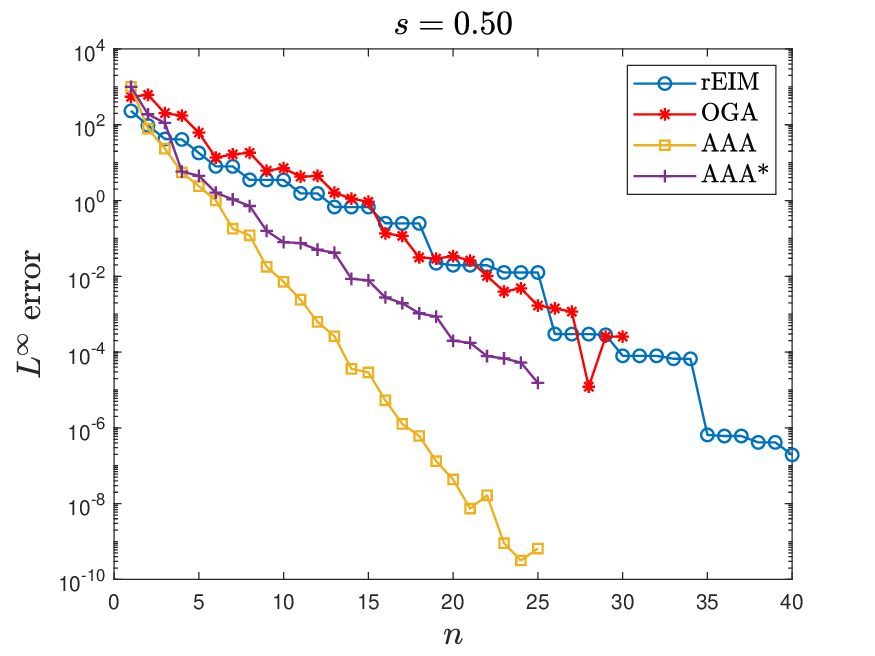}\vspace{12pt} \includegraphics[width=.4\linewidth]{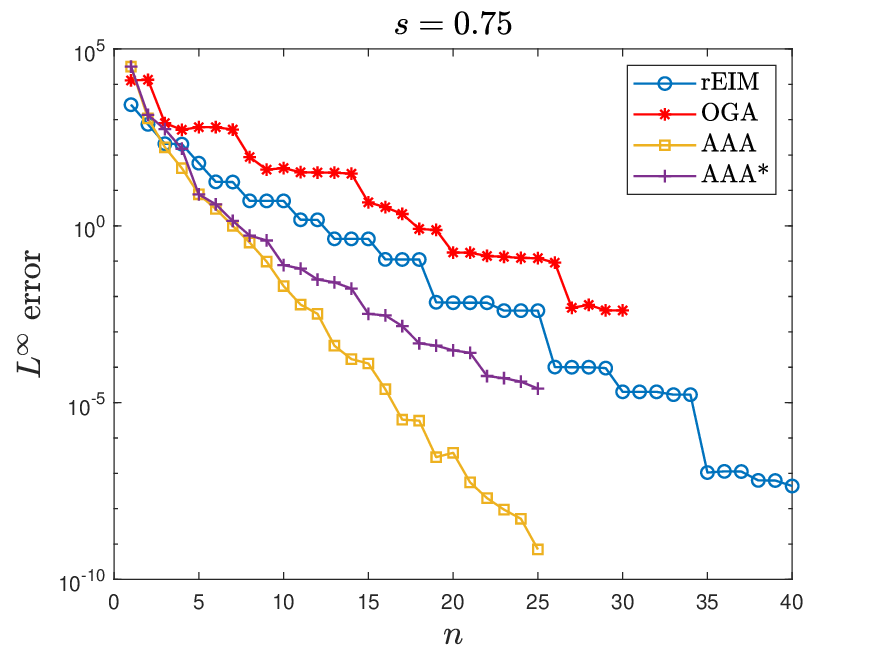}\hspace{8pt}
    \includegraphics[width=.4\linewidth]{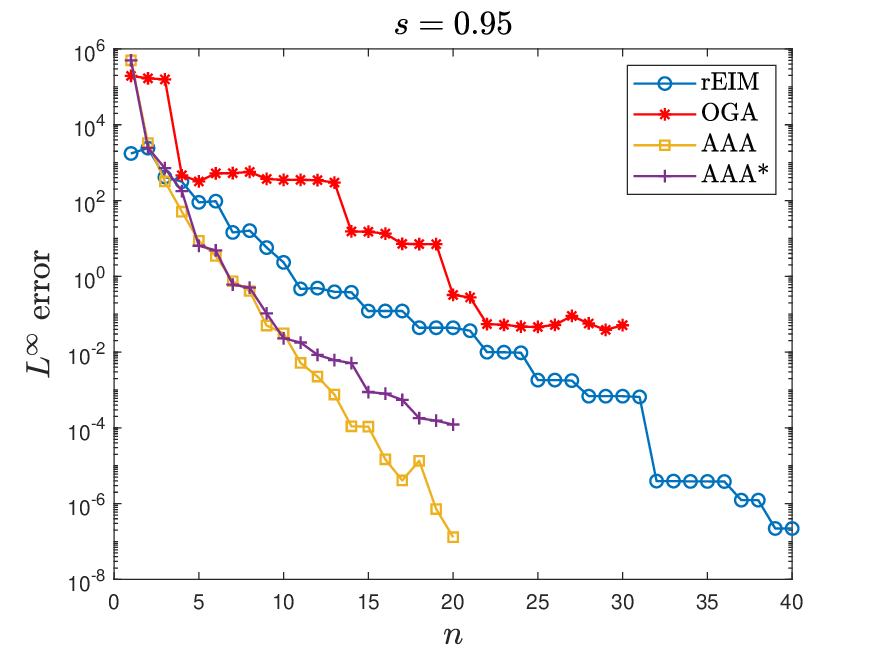}
    \caption{Maximum norm errors of rational approximation algorithms for $x^{-s}$ on $[10^{-6},1]$. }
\label{fig:xalpha}
\end{minipage}
\end{figure}

\subsection{Approximation of Power Functions}\label{subsec:xalpha}
We start with a numerical comparison of the rEIM, the OGA, and the popular AAA rational approximation algorithm for the target function $x^{-s}$ over $[10^{-6},1]$. 
{In this experiment, we have tuned the set of sample points in the AAA algorithm to improve its accuracy as much as possible.} 
Figure \ref{fig:polesort} shows the distribution of sorted poles $-b_1, \ldots, -b_{30}$ and interpolation points $x_1, \ldots, x_{30}$ used in the rEIM. An interesting phenomenon is that the poles and interpolation points are both exponentially clustered at 0. From Figure \ref{fig:polesort} (right), we observe that the Lebesgue constant $L_n$ of $\Pi_n$ grows slowly.


It is shown in  Figure \ref{fig:xalpha} that the AAA algorithm achieves the highest level of accuracy under the same number of iterations. However, the output of AAA is a barycentric representation of rational functions, which should be converted into the form $\sum_{i=1}^n\frac{c_i}{x+b_i}$ {by solving a generalized eigenvalue problem. The composition of AAA and the additional generalized eigen-solver is denoted by AAA*.} Unfortunately, AAA* is significantly less accurate than the original AAA. It is also {observed from} Figure \ref{fig:xalpha} that errors of AAA and AAA* do not decay after 22-25 iterations while the error of the rEIM is finally smaller than AAA*. {In this experiment, the rEIM eventually achieves higher accuracy than the AAA* algorithm under double precision arithmetic.}

\subsection{Fractional Laplacian on Uniform Grids}\label{subsecFLUniform}
On $\Omega = (-1,1)^2$ we consider the fractional Laplacian 
\begin{equation}\label{fractionalLaplacian}
 (-\Delta)^su = 1
\end{equation}
with $u=0$ on $\partial\Omega$.
The reference exact solution is computed by
\begin{equation*}
u\approx\sum_{j,k=1}^{j^2+k^2\leq4\times10^6}\lambda_{jk}^{-s}(1,u_{jk})u_{jk},
\end{equation*}
where $u_{jk}$ is the $L^2$ normalized eigenfunction of $-\Delta$ associated to the eigenvalue $\lambda_{jk}=(j^2+k^2)\pi^2/4$. The fractional Laplacian is reduced by the rEIM in section \ref{subsec:xalpha} with $n=30$, $I=[10^{-6},1]$ to a series of {integer-order} problems \eqref{integerproblem}, which is further solved by finite difference on a uniform grid with mesh size $h_i=2^{-(i+3)},~i=1, ..., 5$. In this case, $\Lambda=10^6$ and $\eta=10^{-6}$ is enough to lower bound $\lambda_{\min}(\mathcal{A}_h)/\Lambda$.  The errors $e_i=\|u-u_{h_i}\|_{L^2(\Omega)}$ for different values of $s$ are recorded in Table \ref{tab:gridwidth} with order of convergence
\begin{equation*}
    \text{order}_{i+1}:= \log\left(\frac{e_{i+1}}{e_i}\right)\bigg{/}\log\left(\frac{h_{i+1}}{h_i}\right),\quad i=1,2,3,4.
\end{equation*} 
In fact, convergence rates in Table \ref{tab:gridwidth} are consistent with the theoretical convergence rate $O(h^{\min\{2,2s+0.5\}})$ in \cite{BonitoPasciak2015}.
 
\begin{table}[thp]
  
  \caption{$L^2$ errors and convergence rates on uniform girds.}
  \resizebox{\columnwidth}{!}{
    \begin{tabular}{|c||c|c||c|c||c|c||c|c|}
    \toprule
    \hline
    \multicolumn{1}{|c||}{\multirow{2}[4]{*}{$h_i$}} & \multicolumn{2}{c||}{$s=0.25$} & \multicolumn{2}{c||}{$s=0.5$} & \multicolumn{2}{c||}{$s=0.75$} & \multicolumn{2}{c|}{$s=0.95$} \\
\cline{2-9}    \multicolumn{1}{|c||}{} & $L^2$ error & order & $L^2$ error & order & $L^2$ error & order & $L^2$ error & order \\
    \hline
    $2^{-4}$ & $9.7461\times 10^{-3}$ &  ---  & $4.8415\times 10^{-3}$
 &  ---  & $2.1959\times 10^{-3}$
 & ---  & $1.2359\times 10^{-3}$ & --- \\
    \hline
    $2^{-5}$ & $4.6362\times 10^{-3}$ & 1.0719 & $1.6187\times 10^{-3}$ & 1.5806 & $5.8485\times 10^{-4}$ & 1.9087 & $3.1211\times 10^{-4}$ & 1.9855 \\
    \hline
   $2^{-6}$ & $2.2817\times 10^{-3}$ & 1.0229 & $5.4426\times 10^{-4}$ & 1.5725 & $1.5303\times 10^{-4}$ & 1.9342 & $7.8298\times 10^{-5}$ & 1.995 \\
    \hline
    $2^{-7}$ & $1.0939\times 10^{-3}$ & 1.0606 & $1.8480\times 10^{-4}$ & 1.5583 & $3.9673\times 10^{-5}$ & 1.9476 & $1.9599\times 10^{-5}$ & 1.9982 \\
    \hline
    $2^{-8}$ & $4.7034\times 10^{-4}$ & 1.2177 & $6.2553\times 10^{-5}$ & 1.5628 & $1.0226\times 10^{-5}$ & 1.9559 & $4.9019\times 10^{-6}$ & 1.9993 \\
    \bottomrule
    \end{tabular}%
    }
  \label{tab:gridwidth}%
\end{table}%

\begin{figure}[thp]
\begin{minipage}{1.0\linewidth}
    \centering
    \includegraphics[width=.32\linewidth]{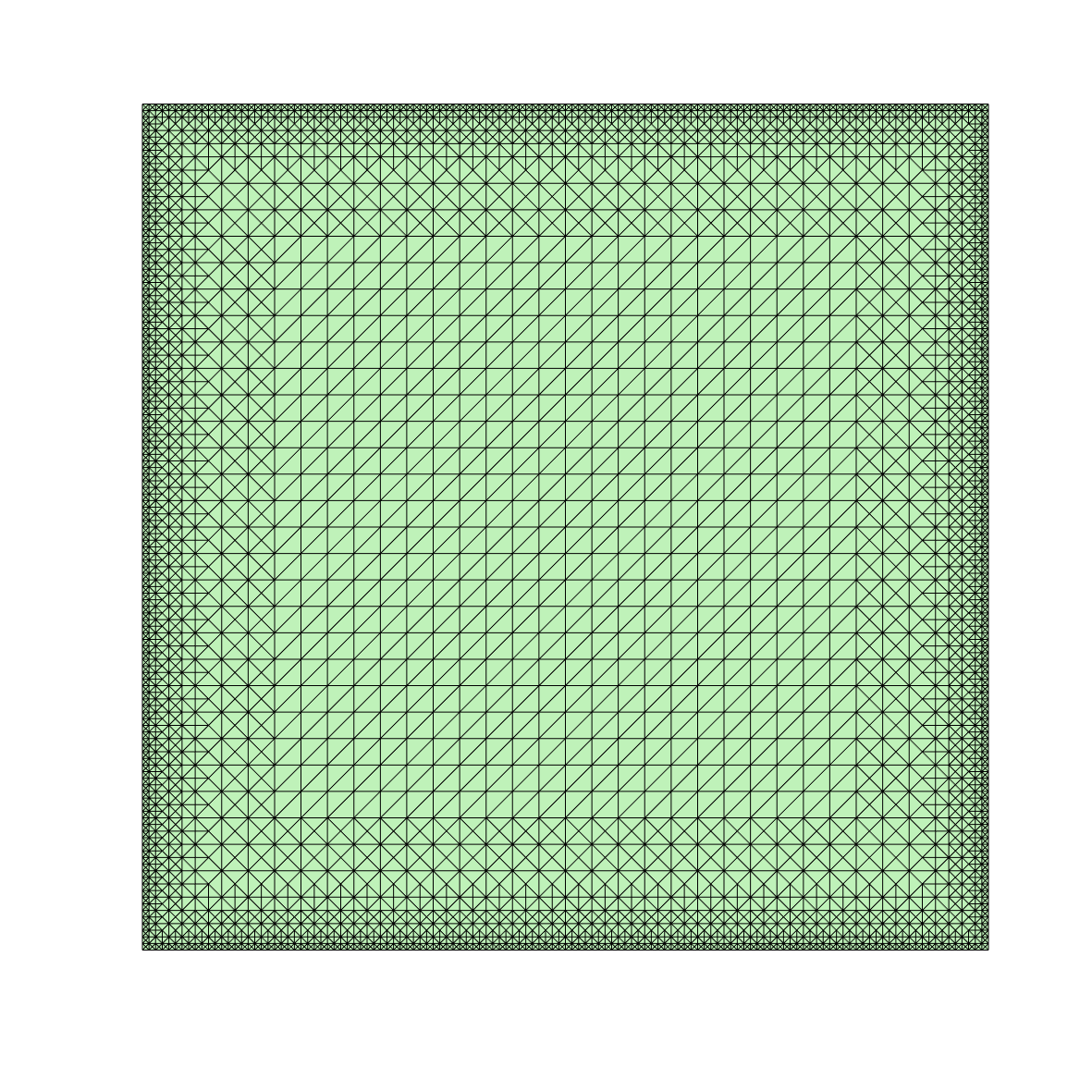}
    \includegraphics[width=.32\linewidth]{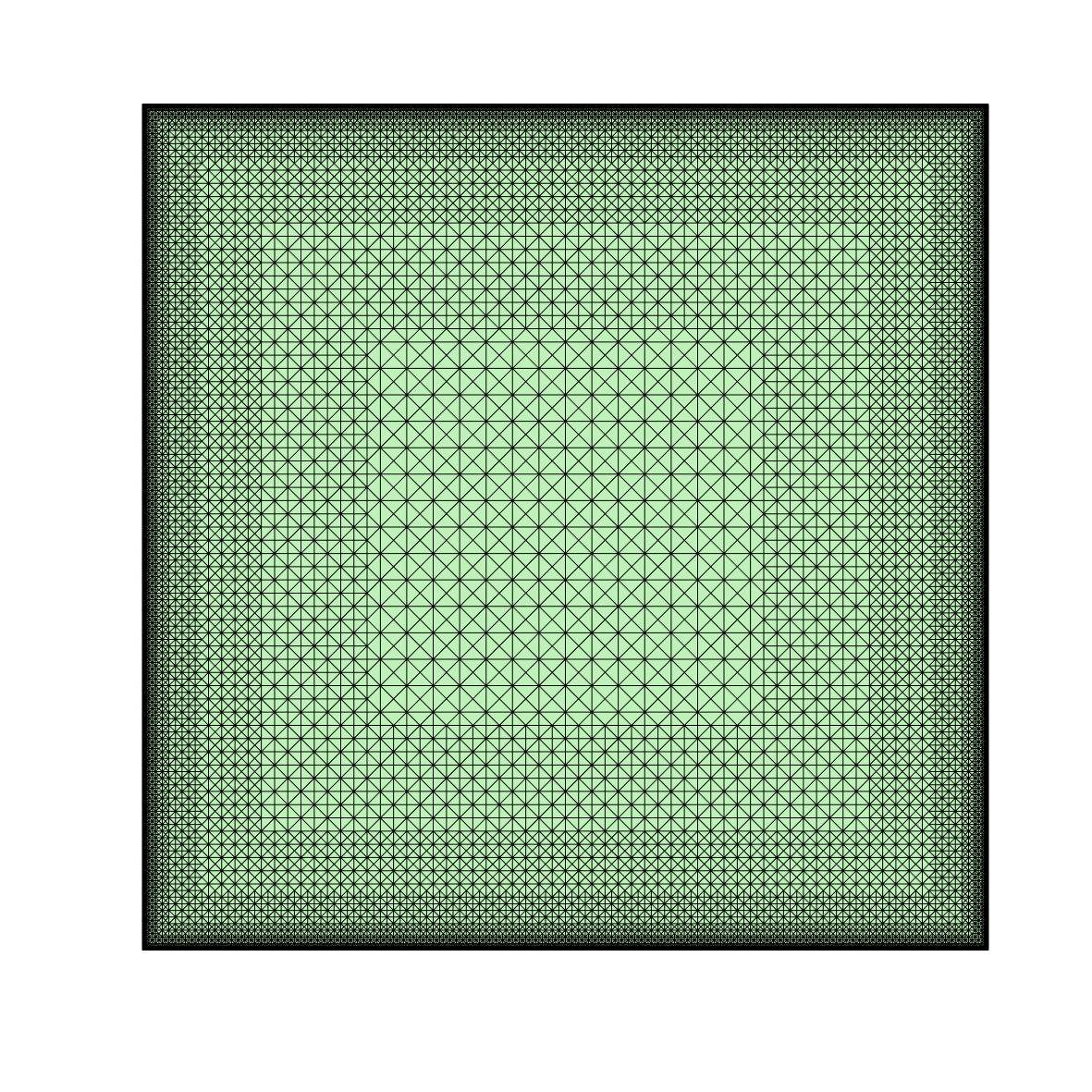} 
    \caption{Graded grid $\mathcal{T}_1$ with 4225 vertices (left); graded grid $\mathcal{T}_2$ with 19585 vertices (right).}
\label{fig:gradedmesh}\end{minipage}
\end{figure} 

\subsection{Fractional Laplacian on Graded Grids}\label{subsecFLgraded} On uniform grids, convergence rates of finite difference $L^2$ errors for fractional Laplacian are slower than $O(h^2)$ when $s<0.75$, see Table \ref{tab:gridwidth}. To improve the numerical accuracy, we test the performance of the rEIM-based solver on a sequence of graded grids designed appropriately to resolve the boundary singularity. It is clear that adaptive mesh refinement is not applicable to rectangular meshes without introducing hanging nodes. Therefore, we discretize the fractional Laplacian \eqref{fractionalLaplacian} with $s=0.25$
by linear finite elements on locally refined triangular meshes. Let $\Omega$ be partitioned by the uniform mesh $\mathcal{T}_0$ with mesh size $h=0.25$ in each direction. Let $N(\mathcal{T})$ denote the number of vertices in $\mathcal{T}$ and $C_T$ the barycenter of a triangle $T$. For $j=1, 2, 3$, we set $\widetilde{\mathcal{T}}_j=\mathcal{T}_{j-1}$ {and} successively mark and refine those triangles $T\in\widetilde{\mathcal{T}}_j$ {that} satisfy $${\rm area}(T)>\frac{6}{N(\widetilde{\mathcal{T}}_j)}\log_{10}(N(\widetilde{\mathcal{T}}_j)){\rm dist}(C_T,\partial\Omega).$$
This loop terminates when $\widetilde{\mathcal{T}}_j$ fulfils  $N(\widetilde{\mathcal{T}}_j)\geq4^j\times10^3$ and we then set $\mathcal{T}_j=\widetilde{\mathcal{T}}_j$, see Figure \ref{fig:gradedmesh} for $\mathcal{T}_1$ and $\mathcal{T}_2$.

Recall that $\mathbb{M}$ and $\mathbb{A}$ are linear finite element stiffness and mass matrices, respectively. The maximum eigenvalue $\lambda_{\max}(\mathcal{A}_h)=\lambda_{\max}(\mathbb{M}^{-1}\mathbb{A})$ on highly contrast meshes grows faster than uniform mesh sequences. Thus we set the eigenvalue upper bound as $\Lambda=10^8$ and generate rEIM rational approximants over $[10^{-8},1]$, see Figure \ref{fig:EIMerrorgraded} (left). In this case, $x^{-s}$ has greater singularity and more rEIM iterations are needed to achieve the same accuracy as in section \ref{subsecFLUniform}. 
From Figure \ref{fig:EIMerrorgraded} (right), we observed that the FEM on graded grids $\{\mathcal{T}_j\}_{1\leq j\leq 4}$ is able to achieve higher-order convergence than the uniform-grid based FEM.

\begin{figure}[thp]
\begin{minipage}{1.0\linewidth}
    \centering
    \includegraphics[width=.4\linewidth]{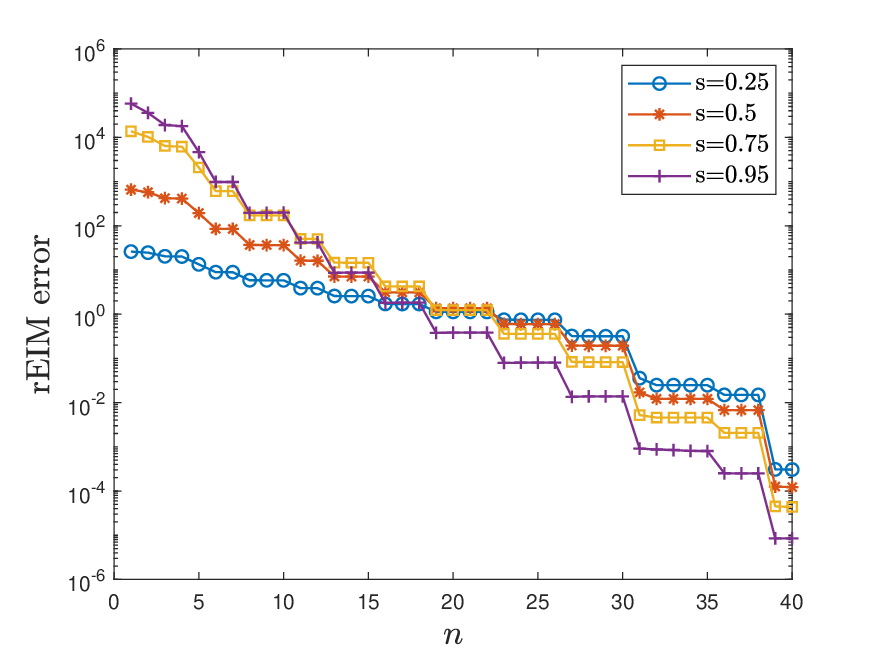}
    \includegraphics[width=.4\linewidth]{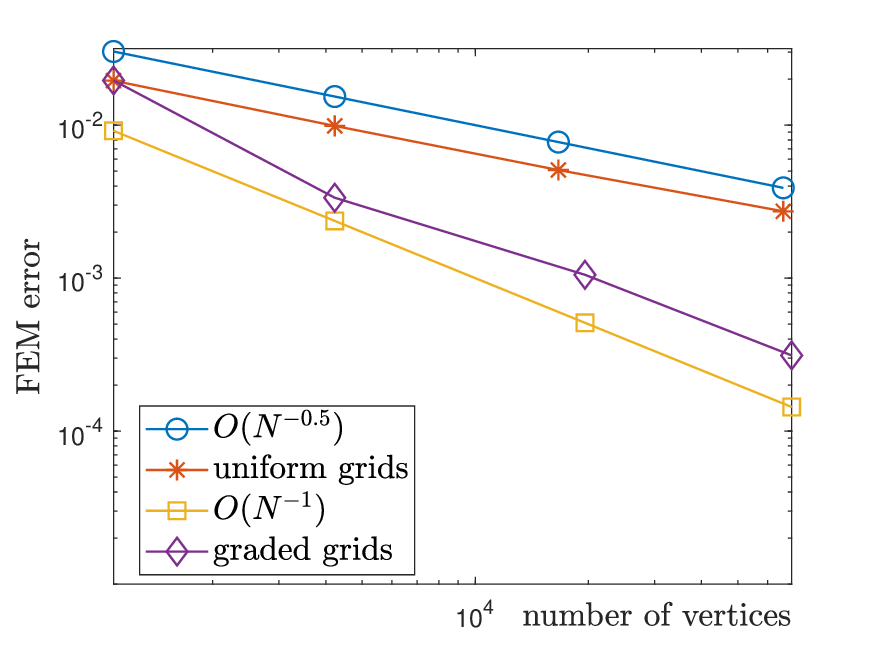} 
    \caption{rEIM $L^\infty$ error for $x^{-s}$ on $[10^{-8},1]$ (left); FEM $L^2$ errors for $s=0.25$, $N$ is the number of grid vertices (right).}
\label{fig:EIMerrorgraded}\end{minipage}
\end{figure} 

\begin{figure}[thp]
\begin{minipage}{1.0\linewidth}
    \centering
    \includegraphics[width=.4\linewidth]{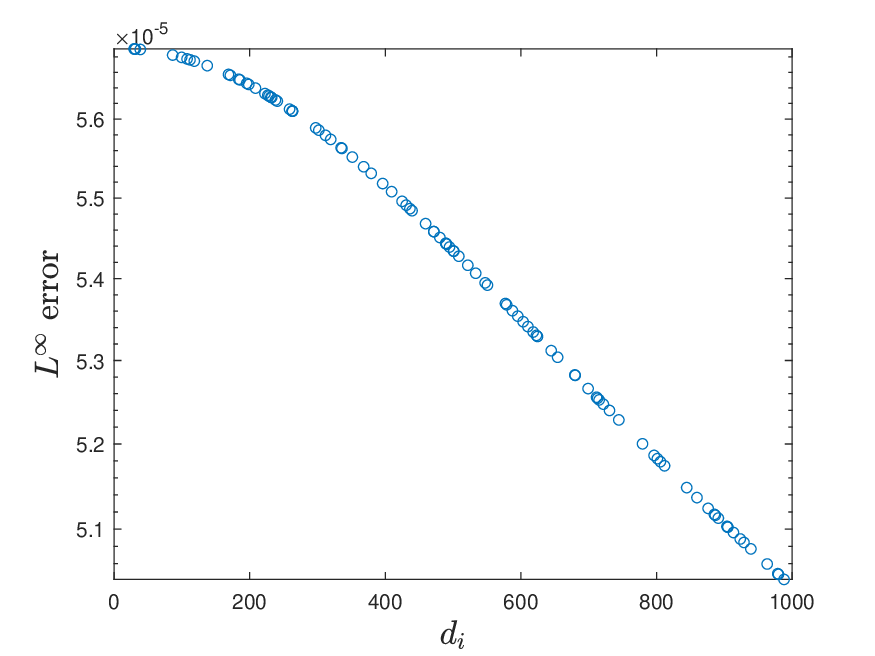}
    \includegraphics[width=.4\linewidth]{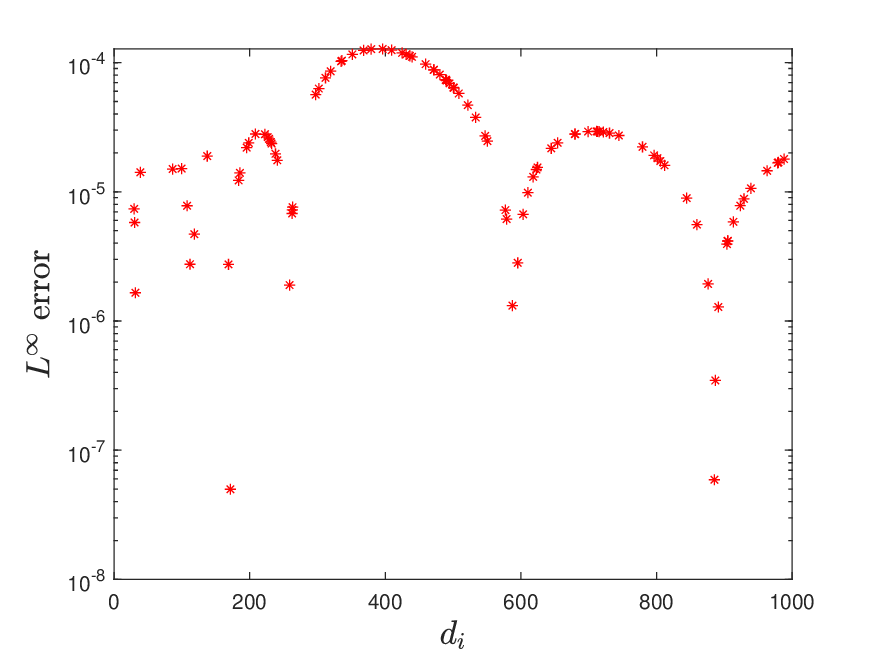} 
    \caption{$L^\infty$ error of the rEIM for $F_{0.5}$ (left) and $F_{1}$ (right).}
\label{fig:EIMerrortime}\end{minipage}
\end{figure} 

\subsection{Adaptive Step-Size Control for Fractional Heat Equations}\label{expfractionalheat}
On $\Omega=(0,1)^2$, we consider the fractional parabolic equation \eqref{heatFL} with $\mathcal{A}=-\Delta$ and the exact solution
\begin{equation}\label{uexact}
    u(t,x,y)={\rm e}^{-t/20}\cos(2\pi t)\sin(\pi x)\sin(\pi y).
\end{equation}
This problem with $s=0.5$ and $s=1$ is solved by the linear FEM on a uniform triangular mesh of mesh size $h=2^{-8}$. The upper bound of eigenvalues of $\mathcal{A}_h$ is  again $\Lambda=10^6$ and $I=[10^{-6},1]$.  Given the error tolerance $tol=10^{-4}$ and $\tau_0=10^{-3}$, we predict a new time step size of the implicit Euler method by the criterion (cf.~\cite{hairer1987solving})
\begin{equation}\label{taunew}
    \tau_{\text{new}} = 0.8\tau_m\left(\frac{tol}{err_m}\right)^{1/2},
\end{equation} 
where $err_m$ is an error estimator described in section \ref{subsec:adaptive}.
If $err_{m+1}\leq tol$, we accept $\tau_{m+1}=\tau_{\text{new}}$ and move forward to $t_{m+1}=t_m+\tau_{m+1}$; otherwise, a new $\tau_{\text{new}}$ is computed by \eqref{taunew} with $m$ replaced by $m+1$.

Recall that we need to approximate $\frac{1}{x^s+1/(\tau_m\Lambda^s)}$ and $\frac{1}{x^s+\kappa_{1,m}/\Lambda^s}$ at each time step. To test the uniform accuracy of the rEIM, we randomly select a point set $\mathcal{S}\subset [1,10^3]$ and consider the function set $$F_s :=\{f_i\in C(I): f_i(x)=(x^{s}+d_i/\Lambda^{s})^{-1},~ d_i\in \mathcal{S}\}.$$ 
The range $[1,10^3]$ contains all possible $1/\tau_m$ and $\kappa_{1,m}$. Figure \ref{fig:EIMerrortime} shows the $L^{\infty}$ interpolation error of the rEIM with 
$n=30$ for the functions in $F_{0.5}$ and $F_{1}$. 
The $L^2$ errors of numerical solutions and the accepted/rejected step sizes are presented in Figure \ref{fig:adaptive}.  In the adaptive process, there are 243 steps and 5 rejected step sizes when $s=0.5$; 238 steps and 6 rejected step sizes when $s=1$.

\begin{figure}[thp]
\begin{minipage}{1.0\linewidth}
   \centering
    \includegraphics[width=.4\linewidth]{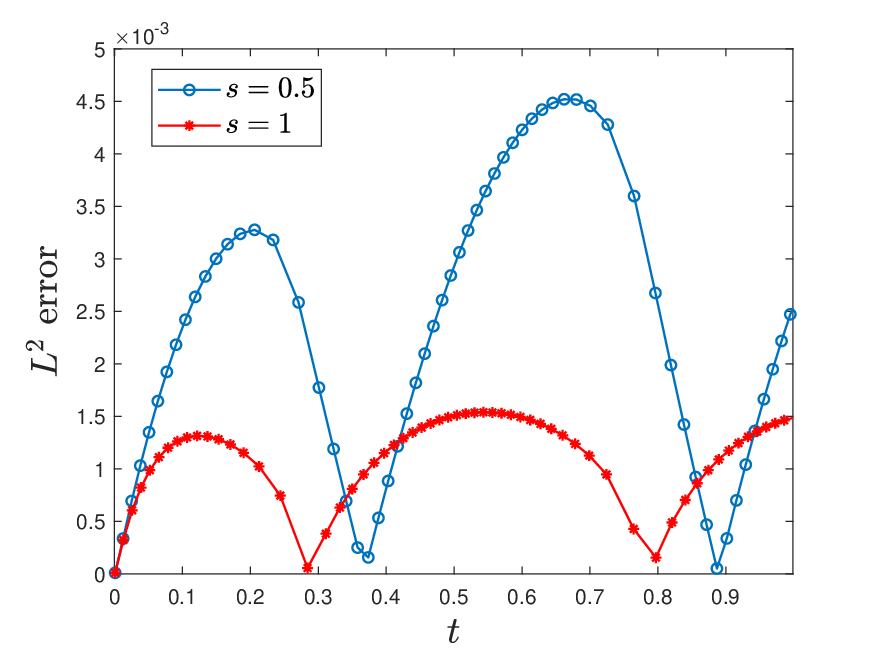}
    \includegraphics[width=.4\linewidth]{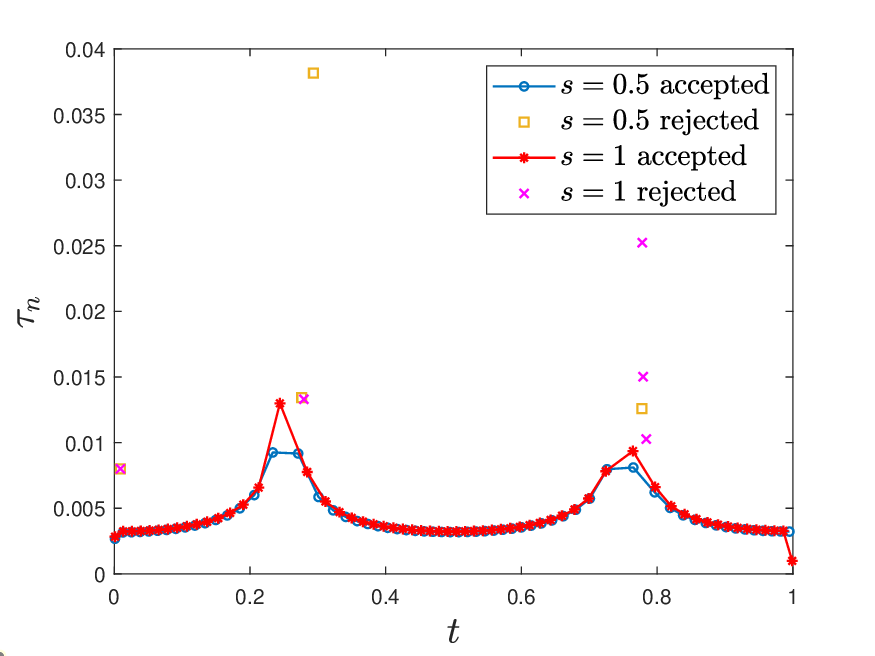}
    \caption{The $L^2$ error of numerical solutions (left); accepted steps and rejected step sizes (right).}
\label{fig:adaptive}
\end{minipage}
\end{figure}

\subsection{Approximation of Other Functions} In the last experiment, we interpolate the functions in sections \ref{subsec:prec} and \ref{subsec:exp} using the rEIM with $n=30$. The left of Figure \ref{fig:other} shows the $L^{\infty}$ error of the rEIM for $(x^{-\frac{1}{2}}+ K_ix^\frac{1}{2})^{-1}$ on $[10^{-6},1]$, where the parameter $K_i$ was randomly selected from [$10^{-6}$,1]. The right of Figure \ref{fig:other} shows the $L^{\infty}$ error of the rEIM for $\exp(-\tau_ix)$ and $\varphi(-\tau_ix)$ on $[1,10^6]$, where the time step size $\tau_i$ was randomly selected from [$0.002$,1]. 

\begin{figure}[thp]
\begin{minipage}{1.0\linewidth}
   \centering
    \includegraphics[width=.4\linewidth]{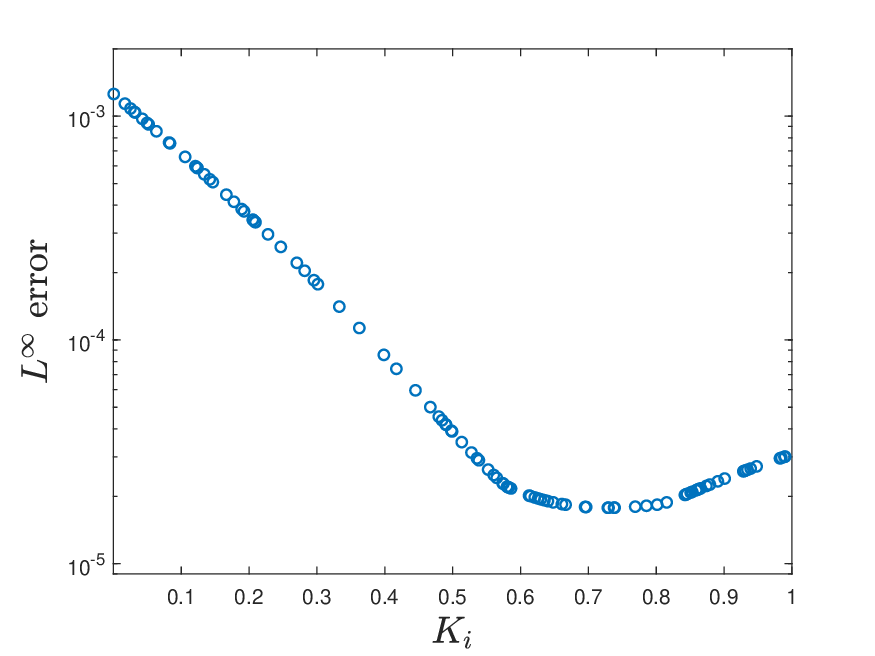}
    \includegraphics[width=.4\linewidth]{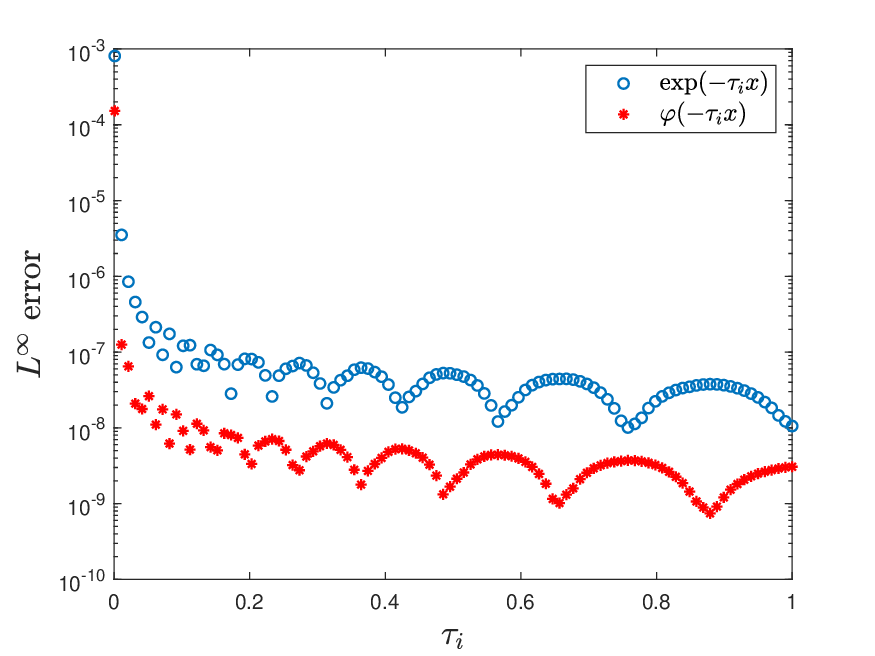}
    \caption{$L^{\infty}$ error for $(x^{-\frac{1}{2}}+ K_ix^\frac{1}{2})^{-1}$ (left); $L^{\infty}$ error for $\exp(-\tau_ix)$ and $\varphi(-\tau_ix)$ (right).}
\label{fig:other}
\end{minipage}
\end{figure}

{\section{Concluding Remarks}
In this paper, we have developed the rEIM, a new rational approximation algorithm for producing partial fraction approximation of a target function set. We  have discussed several applications of the rEIM such as the discretizations of space-fractional elliptic and parabolic equations, robust preconditioning for interface problems, and approximate evaluation of matrix exponentials. In addition, convergence rate of the rEIM is justified based on the entropy numbers of the underlying dictionary. In the future research, we shall investigate other choices of the dictionary $\mathcal{D}$ and possible applications of EIM-type algorithms beyond rational approximation.}

{\section*{Acknowledgments} The author would like to thank the anonymous referees for  many remarks that improve the quality and presentation of this paper.}

\bibliographystyle{siamplain}

\end{document}